\documentclass[10pt,leqno]{amsart}
\usepackage{graphicx}
\usepackage{indentfirst,csquotes}

\usepackage{amssymb}
\usepackage{graphicx}
\usepackage{amssymb}
\usepackage{pifont}
\usepackage{amsmath}
\usepackage{pst-node}
\usepackage{tikz-cd}
\usepackage[margin=3cm]{geometry}
\usepackage{array}
\usepackage{url}
\usepackage{float}
\usepackage[toc,page]{appendix}
\usepackage{MnSymbol}
\usepackage{caption}
\usepackage{diagbox}
\usepackage{mathtools}
\usepackage{scalerel}
\usepackage{stackengine,wasysym}

\newcommand\reallywidetilde[1]{\ThisStyle{%
  \setbox0=\hbox{$\SavedStyle#1$}%
  \stackengine{-.1\LMpt}{$\SavedStyle#1$}{%
    \stretchto{\scaleto{\SavedStyle\mkern.2mu\AC}{.5150\wd0}}{.6\ht0}%
  }{O}{c}{F}{T}{S}%
}}

\theoremstyle{plain}
\newtheorem{theorem}{Theorem}[section]
\newtheorem{proposition}[theorem]{Proposition}
\newtheorem{lemma}[theorem]{Lemma}
\newtheorem{corollary}[theorem]{Corollary}
\newtheorem{conjecture}[theorem]{Conjecture}

\theoremstyle{definition}
\newtheorem{remark}{Remark}
\newtheorem{example}{Example}
\newtheorem{definition}[theorem]{Definition}
\numberwithin{equation}{section}
\numberwithin{table}{section}
\numberwithin{figure}{section}

\DeclarePairedDelimiter\ceil{\lceil}{\rceil}

\begin{document}
	
	\title[The McKay Correspondence for Dihedral Groups]{The McKay correspondence for dihedral groups: The Moduli Space and the Tautological Bundles}
	
	
	\author{John Ashley Capellan}
	\address{Graduate School of Mathematics, Nagoya University, Furo-cho, Chikusa-ku, Nagoya, 464-8602, Japan}
	\curraddr{}
	\email{m20049e@math.nagoya-u.ac.jp}
	\thanks{}
	
	
	\date{}
	
	\dedicatory{}
	
	\commby{}
	
\begin{abstract}
A conjecture posed by Ishii states that for a finite subgroup $G$ of $\operatorname{GL}(2, \mathbb{C})$, a resolution $Y$ of $\mathbb{C}^2/G$ is isomorphic to a moduli space $\mathcal{M}_{\theta}$ of $G$-constellations for some generic stability parameter $\theta$ if and only if $Y$ is dominated by the maximal resolution. This paper affirms the conjecture in the case of dihedral groups as a class of complex reflection groups, and offers an extension of McKay correspondence of Ito-Nakamura for $\operatorname{SL}(2,\mathbb{C})$, and Ishii for small $\operatorname{GL}(2,\mathbb{C})$.
\end{abstract}

	\maketitle

\section{Introduction}

The classical McKay correspondence relates representations of a finite subgroup $G \subset \operatorname{SL}(2, \mathbb{C})$ to the dual graph of exceptional divisors of the minimal resolution of the quotient variety $\mathbb{C}^2/G$.

An algebro-geometric viewpoint of the correspondence was found by \cite{gonzverd} in the case of $G \subset \operatorname{SL}(2, \mathbb{C})$, via some locally free sheaves. In \cite{itonak1} and \cite{itonak2}, these locally free sheaves are realized as tautological bundles. An explicit description is obtained using the $G$-Hilbert scheme $G\operatorname{-Hilb}(\mathbb{C}^2)$ as the minimal (crepant) resolution of the quotient variety $\mathbb{C}^2/G$. The McKay correspondence of $\operatorname{SL}(2, \mathbb{C})$ was obtained by computing the minimal generators of the $G$-module $I/(\mathfrak{m} I + \mathfrak{n})$ of each $G$-cluster. This computation is related to the tops and socles of each $G$-cluster (which is defined in Definition \ref{topsoc}).

The aforementioned viewpoints of the correspondence (both \cite{gonzverd} and \cite{itonak1}) give a rigorous proof of the correspondence between representations of $G \subset \operatorname{SL}(2)$ and the exceptional divisors of the minimal resolution of $\mathbb{C}^2/G$. This correspondence can be generalized further into equivalences of derived categories.

A natural generalization of the McKay correspondence is an equivalence between the $G$-equivariant geometry of $\mathbb{C}^n$ and the geometry of a crepant resolution $Y$ of $\mathbb{C}^n/G$ expressed in the language of derived categories. When $Y \subset G\operatorname{-Hilb}(\mathbb{C}^n)$ is the irreducible component of $G\operatorname{-Hilb}(\mathbb{C}^n)$ which contains the open subset of all reduced $G$-clusters, the celebrated result of \cite{bridgelandkingreid} states some conditions that will determine that $\tau: Y \rightarrow \mathbb{C}^n/G$ is a crepant resolution and that $\Phi: D(Y) \rightarrow D^G(\mathbb{C}^n)$ given by the Fourier-Mukai transform gives a derived equivalence. More particularly and explicitly, for $n \leq 3$ and $G \subset \operatorname{SL}(n,\mathbb{C})$, $\Phi$ defines a derived equivalence.

However, the story does not stop here. The moduli space of $G$-clusters provides one candidate for a crepant resolution of the quotient variety $\mathbb{C}^n/G$. With the conjecture of Reid (Conjecture 4.1 of \cite{reid}) in mind, which states that if $\tau: Y \rightarrow \mathbb{C}^n/G$ is a crepant resolution, then $\Phi: D(Y) \xrightarrow{\sim} D^G(\mathbb{C}^n)$ for some derived equivalence $\Phi$, there is an ongoing search for such crepant resolutions. One of such candidates is the moduli space of $G$-constellations.

A generalization of Hilbert scheme of $G$-orbits is the moduli space of $G$-constellations (on an affine space) which are introduced in \cite{crawishii}. The moduli space depends on some stability parameter $\theta$ and the moduli space of $\theta$-stable $G$-constellations is denoted by $\mathcal{M}_{\theta}$. If $G$ is a subgroup of $\operatorname{SL}(n, \mathbb{C})$ acting on $\mathbb{C}^n$ and $n \leq 3$, then $\mathcal{M}_{\theta}$ is a crepant resolution of $\mathbb{C}^n/G$ for a generic stability parameter $\theta$. The main results in \cite{crawishii}, \cite{yamagishiabelian}, \cite{yamagishinonabelian} realize a (projective) crepant resolution $Y$ of $\mathbb{C}^3/G$ (for any finite subgroup $G \subset \operatorname{SL}(3, \mathbb{C})$) as a moduli space $\mathcal{M}_{\theta}$ of $G$-constellations for some generic stability parameter $\theta$. More precisely, there is a generic stability parameter $\theta$ such that $Y \cong \mathcal{M}_{\theta}$.

If we generalize from $\operatorname{SL}(2, \mathbb{C})$ to $\operatorname{GL}(2, \mathbb{C})$, which can be either small (i.e. which does not contain pseudoreflections) or non-small, we get a more general McKay correspondence.

The paper \cite{wunram} obtained an algebraic-geometric viewpoint of the correspondence in the case of small subgroups of $\operatorname{GL}(2, \mathbb{C})$. Other explicit descriptions of the correspondence, by derived functors and by the work of \cite{wunram}, were obtained in \cite{art1ishii}, using the $G$-Hilbert scheme $G\operatorname{-Hilb}(\mathbb{C}^2)$ as the minimal resolution of the quotient variety $\mathbb{C}^2/G$, and determined their tops and socles of each $G$-cluster recovering the Ito-Nakamura type of correspondence, i.e. the socles of a $G$-cluster recover the same $G$-modules as with the module $I/(\mathfrak{m} I + \mathfrak{n})$.

Since $\mathcal{M}_{\theta}$ is a resolution of $\mathbb{C}^2/G$, there is a fully faithful functor from $D(\mathcal{M}_{\theta}) \hookrightarrow D^G(\mathbb{C}^2)$. In relation to the DK hypothesis and the maximal resolution in \cite{kawamata}, Ishii posed the following conjecture (3.1) in \cite{art2ishii}:

\begin{conjecture}
	Let $G \subset \operatorname{GL}(2, \mathbb{C})$ be a finite subgroup and consider the quotient $X = \mathbb{C}^2/G$ with the boundary divisor $B$. For any resolution of singularities $\tilde{Y} \rightarrow X$, $\tilde{Y}$ is isomorphic to $\mathcal{M}_{\theta}$ for some generic stability parameter $\theta$ if and only if $\tilde{Y}$ is between the minimal and maximal resolution of $\mathbb{C}^2/G$, where the maximal resolution means the smooth variety which has unique maximal coefficients satisfying the inequality in Definition \ref{maxdef}.
\end{conjecture}

So far, this conjecture is solved in the cases of abelian subgroups and small subgroups of $\operatorname{GL}(2, \mathbb{C})$.

It is natural to ask if it is possible to formulate the explicit descriptions of the McKay correspondence in the case of complex reflection groups. This case is particularly interesting because the quotient variety $\mathbb{C}^2/G$ is isomorphic to $\mathbb{C}^2$ itself in which the minimal resolution is the identity map which reveals no data about the exceptional divisors. Hence, we consider the aforementioned maximal resolution in the hopes of recovering a McKay correspondence. In this paper, we offer an explicit description of the McKay correspondence in the case of dihedral groups via its derived equivalence.

\textit{Notation:} Let $D_{2n} = \left\langle \tau :=  \begin{bmatrix}
	0 & 1 \\
	1 & 0
\end{bmatrix}, \sigma := \begin{bmatrix}
	\epsilon & 0 \\
	0 & \epsilon^{-1}
\end{bmatrix} (\epsilon^n = 1)  \right\rangle \subset \operatorname{GL}(2,\mathbb{C})$ be a dihedral group of order $2n$ embedded in the general linear group $\operatorname{GL}(2, \mathbb{C})$. Unless explicitly stated, $G$ is the dihedral group $D_{2n}$. The representations of $D_{2n}$ are $\rho_0$ which is the trivial representation; $\rho_0', \rho_{n/2}, \rho_{n/2}'$ are non-trivial $1$-dimensional representations; and $\rho_j (j \neq 0, 0', n/2, n/2')$ are $2$-dimensional representations. The character table is as follows: 

\begin{table}[H]
	\centering
	\caption{The Irreducible Representations of $D_{2n}$}
	\label{tab:my_label}
	\begin{tabular}{c|c|c|c}
		\backslashbox{\mbox{Representation}}{\mbox{Conjugacy Class}}  & $1$ & $\tau$ & $\sigma^i$ \\\hline
		$\rho_0$ & $1$ & $1$ & $1$\\
		$\rho_0'$ & $1$ & $-1$ & $1$\\
		$\rho_j$ & $2$ & $0$ & $\epsilon^{ij} + \epsilon^{-ij}$\\
		$\rho_{n/2}$ \mbox{($n$ even)} & $1$ & $1$ & $(-1)^i$ \\
		$\rho_{n/2}'$ \mbox{($n$ even)} & $1$ & $-1$ & $(-1)^i$
	\end{tabular}
	
	where $1 \leq j \leq \frac{n-1}{2}$ and $0 \leq i \leq n-1$ are integers.
\end{table}

The main results of this paper are the following:

\begin{theorem}[Theorem \ref{maxemb}]
	The maximal resolution $Y_{max}$ of $(\mathbb{C}^2/G, \hat{B})$, defined as the smooth variety which has unique maximal coefficients satisfying the inequality in Definition \ref{maxdef}, is isomorphic to the quotient variety 
	
	$\mathbb{Z}_n\operatorname{-Hilb}(\mathbb{C}^2)/\mathbb{Z}_2 := \langle \sigma \rangle\operatorname{-Hilb}(\mathbb{C}^2)/(D_{2n}/\langle \sigma \rangle)$, where $\hat{B}$ is a $\mathbb{Q}$-divisor defined by the equation $K_{\mathbb{C}^2} = \pi^*(K_{\mathbb{C}^2/D_{2n}} + \hat{B})$, and $\pi: \mathbb{C}^2 \rightarrow \mathbb{C}^2/D_{2n}$ is the projection map. It is also the minimal embedded resolution of $(\mathbb{C}^2/G, \hat{B})$.
\end{theorem}

The minimal embedded resolution of $(\mathbb{C}^2/G, \hat{B})$ is the smooth surface obtained after the least number of monoidal transformations such that the strict transform of (the support of) $\hat{B}$ is smooth.

\begin{theorem}[Theorem \ref{maxmain}]
	A resolution of singularities $Y \rightarrow \mathbb{C}^2/D_{2n}$ is isomorphic to $\mathcal{M}_{\theta}$ for some generic $\theta$ if and only if $Y$ is dominated by the maximal resolution of the pair $(\mathbb{C}^2/D_{2n}, \hat{B})$.
\end{theorem}

\begin{remark}\label{resmoth}
	We point out here that $\mathbb{C}^2/D_{2n}$ is non-singular. The notion of resolution $Y \xrightarrow{f} \mathbb{C}^2/D_{2n}$ is a proper birational morphism from the smooth variety $Y$.
\end{remark}

Section 3 is devoted to the proofs of the first two theorems. The aforementioned theorems are proved by explicitly tracing the by-product of the quotient variety $\mathbb{C}^2/G \cong G\operatorname{-Hilb}(\mathbb{C}^2)$ as a subscheme in $G\operatorname{-Hilb}(\mathbb{C}^3)$ under flopping operations via \cite{nollasekiya}, which will show that all of the two-dimensional counterparts can be realized as a moduli space of $D_{2n}$-constellations $\mathcal{M}_{\theta}$ for some generic stability parameter $\theta$. Both abstract and explicit (via open set) approaches are presented here.

In the last two sections of this article, we offer two lenses of establishing the McKay correspondence: from the tautological bundles and from the tops and socles of $D_{2n}$-constellation. In particular, for both lenses, we apply the results shown in Section 3 that the maximal resolution can be realized as a moduli space of $D_{2n}$-constellations. 

In Section 4, we use the aforementioned results to construct a similar McKay correspondence from the work of \cite{artverd}. We also show some defects of considering tautological vector bundles over the coarse moduli space, making the quotient stack the better venue to establish the correspondence.

We define the stack associated to the maximal resolution $\mathcal{Y} := [Y_{max}]$ realized as the $2$nd root stack 
$$\mathcal{Y} = \sqrt{(O_{Y_{max}}(\mathbb{B}), 1_{\mathbb{B}})/(Y_{max})},$$
where $B$ is the boundary divisor or the strict transform of $\hat{B}$, $(f_2)_*^{-1}(\hat{B})$, under the maximal resolution $f_2: Y_{max} \rightarrow \mathbb{C}^2/G$ and $\mathbb{B}:= \ceil*{B}$ (for prime divisors $B_{\alpha}$ of $B = \Sigma_{\alpha} b_{\alpha} B_{\alpha}$, $\ceil*{B} := \ceil*{b_{\alpha}}B_{\alpha}$); the global section $1_{\mathbb{B}}$ is induced by the inclusion of divisors $O_{Y_{max}} \hookrightarrow O_{Y_{max}}(\mathbb{B})$. 

We refer to Section 2.2 of \cite{cadman} for the detailed definition of the root stack.
Explicitly, the objects of $\mathcal{Y}$ over a scheme $S$ are quadruples $(f,M,t,\phi)$, where $f: S \rightarrow Y_{max}$ is a morphism, $M$ is an invertible sheaf on $S$, $t \in \Gamma(S,M)$, and $\phi: M^{\otimes 2} \rightarrow f^*(O_{Y_{max}}(\mathbb{B}))$ is an isomorphism such that $\phi(t^2) = f^*(1_{\mathbb{B}})$.

By Theorem \ref{maxemb}, we obtain the isomorphism between stacks: $$\mathcal{Y} \cong [\mathbb{Z}_n\operatorname{-Hilb}(\mathbb{C}^2)/\mathbb{Z}_2].$$
We have the Fourier-Mukai transforms (defined in Section 4):
\begin{align*}
	\Phi: D([\mathbb{C}^2/D_{2n}]) &\rightarrow D(\mathcal{Y}) \\
	\mathfrak{G} &\mapsto \phi_*(Rp_{[\mathbb{Z}_n\operatorname{-Hilb}(\mathbb{C}^2)/D_{2n}]*} (p_{[\mathbb{C}^2/D_{2n}]}^*(\mathfrak{G}) \otimes O_{[\mathcal{Z}/D_{2n}]}))\\
	\Psi: D(\mathcal{Y}) &\rightarrow D([\mathbb{C}^2/D_{2n}]) \\
	\epsilon &\mapsto R(p_{[\mathbb{C}^2/D_{2n}]*}) (p_{[\mathbb{Z}_n\operatorname{-Hilb}(\mathbb{C}^2)/D_{2n}]}^*(\phi^*(\epsilon)) \otimes det(\rho_{nat}) \\
	& \otimes O_{[\mathcal{Z}/D_{2n}]}^{\lor}[2])
\end{align*}

The functors $\Psi$ and $\Phi$ are equivalences via Theorem 4.1 of \cite{ishiiueda}.

We define the tautological sheaf associated to a representation $\rho$ of $D_{2n}$ as $$\hat{\mathcal{R}}_{\rho} := \Phi(O_{\mathbb{C}^2} \otimes \rho^{\lor}).$$

\begin{theorem}[Theorem \ref{mainthm2}]
	The tautological bundles on the stack $\mathcal{Y}$ are described by the following: 
	\begin{center}
		\captionof{table}{The Tautological Bundles for odd $n$}
		\begin{tabular}{|c c c|} 
			\hline
			Tautological Sheaf & Description & Chern Class \\ [0.5ex] 
			\hline
			$\hat{\mathcal{R}}_{\rho_0}$ & $\mathcal{O}_{\mathcal{Y}}$ & $0$ \\ 
			\hline
			$\hat{\mathcal{R}}_{\rho_0'}$ & $\mathcal{O}_{\mathcal{Y}}(\mathcal{B}_3 - \mathcal{D})$ & $\mathcal{B}_3 - \mathcal{D}$ \\
			\hline
			$\hat{\mathcal{R}}_{\rho_i}$ (rank 2) & $0 \rightarrow \mathcal{O}_{\mathcal{Y}} \xrightarrow{i} \hat{\mathcal{R}}_{\rho_i} \xrightarrow{pr} \mathcal{O}_{\mathcal{Y}}(\mathcal{D}_i + \mathcal{B}_3 - \mathcal{D}) \rightarrow 0$ & $\mathcal{D}_i + \mathcal{B}_3 - \mathcal{D}$ \\
			\hline
		\end{tabular}
		
		\captionof{table}{The Tautological Bundles for even $n$}
		\begin{tabular}{|c c c|} 
			\hline
			Tautological Sheaf & Description & Chern Class \\ [0.5ex] 
			\hline
			$\hat{\mathcal{R}}_{\rho_0}$ & $\mathcal{O}_{\mathcal{Y}}$ & $0$ \\ 
			\hline
			$\hat{\mathcal{R}}_{\rho_0'}$ & $\mathcal{O}_{\mathcal{Y}}(\mathcal{B}_1 -\mathcal{B}_2)$ & $\mathcal{B}_1 -\mathcal{B}_2$ \\
			\hline
			$\hat{\mathcal{R}}_{\rho_i}$ (rank 2) & $0 \rightarrow \mathcal{O}_{\mathcal{Y}} \xrightarrow{i} \hat{\mathcal{R}}_{\rho_i} \xrightarrow{pr} \mathcal{O}_{\mathcal{Y}}(\mathcal{D}_i + \mathcal{B}_1 - \mathcal{B}_2) \rightarrow 0$ & $\mathcal{D}_i + \mathcal{B}_1 - \mathcal{B}_2$ \\
			\hline
			$\hat{\mathcal{R}}_{\rho_{(n/2)}}$ & $\mathcal{O}_{\mathcal{Y}}(\mathcal{B}_1)$ & $\mathcal{B}_1$ \\
			\hline
			$\hat{\mathcal{R}}_{\rho_{(n/2)}'}$ & $\mathcal{O}_{\mathcal{Y}}(\mathcal{B}_2)$ & $\mathcal{B}_2$ \\ [1ex] 
			\hline
		\end{tabular}
	\end{center}
	where $\pi: \mathcal{Y} \rightarrow Y_{max}$ is the morphism to the coarse moduli space,
	
	$p: \mathbb{Z}_n\operatorname{-Hilb}(\mathbb{C}^2) \rightarrow \mathbb{Z}_n\operatorname{-Hilb}(\mathbb{C}^2)/\mathbb{Z}_2$ is the projection,
	
	$\mathcal{B}_i$ is a prime divisor on $\mathcal{Y}$ such that $2\mathcal{B}_i = \pi^{-1}(B_i)$ (the stacky locus),
	
	$\mathcal{D} := \pi^{-1}(D)$, where $D$ is a prime divisor of $Y_{max}$ that satisfies all of the following properties: (1) $D$ does not coincide with $B_i$ for all $i$, (2) $D$ is transversal to the exceptional divisor intersecting $B_1$ and $B_2$ (or $B_3$), and (3) $D \cdot E_j = 0, j \neq m$, and 
	
	$\mathcal{D}_i := \pi^{-1}(D_i)$ with $D_i := p(\tilde{D}_i + g \cdot \tilde{D}_i)$, whose $\tilde{D}_i$ is a transversal divisor to an exceptional divisor $\tilde{E}_i$ of the minimal resolution $\mathbb{Z}_n\operatorname{-Hilb}(\mathbb{C}^2) \rightarrow \mathbb{C}^2/\mathbb{Z}_n$. (The abuse of notation is due to Theorem \ref{maxemb}.)
	
	The rank one tautological bundles on the stack are uniquely determined by their Chern classes; and the rank two tautological bundles are determined by an extension of two line bundles. Furthermore, there is only one possible (non-trivial) extension class, making these descriptions unique.
\end{theorem}

In the final section, building from the works \cite{itonak1}, \cite{itonak2} and \cite{art1ishii}, we formulate an analogous description of the socles of the $G$-constellations over exceptional divisors on the stack. We mainly apply the derived equivalence of \cite{art1ishii} in the computation of the socles over the quotient stack. Once again, the deficiency of working over the coarse moduli space appears once again failing to separate the 1-dimensional representations. We culminate in the following descriptions:

\begin{theorem}[Theorem \ref{mainthm3}]
	For a given $D_{2n}$-constellation $\overline{F}_{st}$ on the exceptional divisors over the quotient stack $\mathcal{Y}$, where the exceptional divisors $\mathcal{E}_i$ satisfies $\pi^{-1}(E_i) = \mathcal{E}_i$,
	\[ \operatorname{top}(\overline{F}_{st}) = \rho_0 \oplus \rho_0' \]
	\[ \operatorname{socle}(\overline{F}_{st}) = \begin{cases} 
		\rho_i & [\overline{F}_{st}] \in \mathcal{E}_i, [\overline{F}_{st}] \notin \mathcal{E}_j, i \neq j \\
		\rho_i \oplus \rho_j & [\overline{F}_{st}] \in \mathcal{E}_i \cap \mathcal{E}_j \\
		\rho_{\frac{n}{2} - 1} \oplus \rho_{n/2} \oplus \rho'_{n/2} & [\overline{F}_{st}] \in \mathcal{E}_{\frac{n}{2} - 1} \cap \mathcal{E}_{n/2} \mbox{ ($n$ even)}\\
		\rho_{n/2} \oplus \rho'_{n/2} & [\overline{F}_{st}] \in \mathcal{E}_{n/2} - (\mathcal{E}_{\frac{n}{2} - 1} \cup \{\mathcal{B}_1, \mathcal{B}_2\}) \mbox{ ($n$ even)}\\
		\rho'_{n/2} & [\overline{F}_{st}] = \mathcal{B}_1 \mbox{ ($n$ even)}\\
		\rho_{n/2} & [\overline{F}_{st}] = \mathcal{B}_2 \mbox{ ($n$ even)}\\
	\end{cases}
	\]
\end{theorem}

\begin{remark}
	It is imperative to comment about the subject of dihedral groups. The binary dihedral $\operatorname{SL}(2)$ case has a well-established McKay correspondence, especially that a minimal resolution is given by the $BD_{2n}$-Hilbert scheme in Section 13 of \cite{itonak2}. The Riemenschneider \cite{reimenschneider} notation $D_{n,q}$, which also appeared in \cite{wunram}, is mainly defined for small dihedral groups. These are independent to the results for the dihedral group as a reflection group. 
\end{remark}

\section{Preliminaries}
\subsection{$G$-constellations on $\mathbb{C}^n$}
\subsubsection{Definitions}
Let $V = \mathbb{C}^n$ be an affine space and $G \subset \operatorname{GL}(V)$ be a finite subgroup.
\begin{definition}
	A $G$-constellation on $V$ is a $G$-equivariant coherent sheaf $E$ on $V$ such that $H^0(E)$ is isomorphic to the regular representation of $G$ as a $\mathbb{C}[G]$-module. In symbols, $H^0(E) \cong \mathbb{C}[G]$.
\end{definition}

\begin{example}
	When $E = O_Z$, the structure sheaf $O_Z$ of a $G$-cluster $Z$ inside $V$, $E$ is a $G$-constellation.
\end{example}

Let $R(G) = \oplus_{\rho \in \operatorname{Irr}(G)} \mathbb{Z}\rho$ be the representation ring of $G$, where $\operatorname{Irr}(G)$ denotes the set of irreducible representations of $G$. The parameter space of stability conditions of $G$-constellations is the $\mathbb{Q}$-vector space:
$$\Theta := \{\theta \in \operatorname{Hom}_{\mathbb{Z}}(R(G), \mathbb{Q}) | \theta(\mathbb{C}[G]) = 0 \},$$
where $\mathbb{C}[G]$ is regarded as the regular representation of $G$.

\begin{definition}
	Given $\theta \in \Theta$, a $G$-constellation $E$ is \textbf{$\theta$-stable} (resp. \textbf{$\theta$-semistable}) if every proper $G$-equivariant coherent subsheaf $0 \subsetneq F \subsetneq E$ satisfies $\theta(H^0(F)) > 0$ (resp. $\theta(H^0(F)) \geq 0$). We regard $H^0(F)$ here as an element of $R(G)$.
\end{definition}

\begin{definition}
	A parameter $\theta \in \Theta$ is \textbf{generic} if a $G$-constellation which is $\theta$-semistable is also $\theta$-stable.
\end{definition}

By Proposition 5.3 of \cite{king}, there is a fine moduli scheme $\mathcal{M}_{\theta} = \mathcal{M}_{\theta}(V)$ of $\theta$-stable $G$-constellations on $V$ for generic $\theta$.

There is a morphism $\tau: \mathcal{M}_{\theta}(V) \rightarrow V/G$ which sends a $G$-constellation to its support. By Proposition 2.2 of \cite{crawishii}, $\tau$ is a projective morphism when $\theta$ is generic.

\begin{definition}
	The subset $\Theta^{gen} \subset \Theta$ of generic parameters is open and dense. It is the disjoint union of finitely many convex polyhedral cones $C$ in $\Theta$ (see Lemma 3.1 of \cite{crawishii}). The convex polyhedral cone $C$ is called a \textit{chamber} in $\Theta$.
\end{definition}

For $\theta \in \Theta^{gen}$, the moduli space $\mathcal{M}_{\theta}$ only depends on the open Geometric Invariant Theory (GIT) chamber $C \subset \Theta$ containing $\theta \in \Theta$, so that we can write $\mathcal{M}_C$ instead of $\mathcal{M}_{\theta}$ for any $\theta \in C$. The following theorem gives an example:

\begin{theorem}[Theorem 1.1 of \cite{crawishii}, Theorem 1.1 of \cite{yamagishinonabelian}]
	For a finite subgroup $G \subset \operatorname{SL}(3, \mathbb{C})$, suppose that $Y \rightarrow \mathbb{C}^3/G$ is a projective crepant resolution. Then $Y \cong \mathcal{M}_C$ for some GIT chamber $C \subset \Theta$.
\end{theorem}

The following theorem describes the structure of $G$-constellations for $n = 2$. The arguments on Theorems 1.1 and 1.2 in \cite{bridgelandkingreid} can be adapted to guarantee not only a resolution of singularities of $\mathbb{C}^2/G$, but also the embedding of their corresponding derived categories, which tell the relationship between canonical divisors via inequalities following the DK-hypothesis in \cite{kawamata}.

\begin{proposition}[Theorem 3 of \cite{art2ishii}]
	Let $G$ be a finite subgroup of $\operatorname{GL}(2, \mathbb{C})$. If $\theta$ is generic, then the moduli space $\mathcal{M}_{\theta}$ is a resolution of singularities of $\mathbb{C}^2/G$. Moreover, the universal family of $G$-constellations defines a fully faithful functor $$\Phi_{\theta}: D^b(coh(\mathcal{M}_{\theta})) \rightarrow D^b(coh^G(\mathbb{C}^2)).$$
\end{proposition}

\subsection{The Maximal Resolution}
\label{groupprojection}
Let $G$ be a finite subgroup of $\operatorname{GL}(2, \mathbb{C})$, not necessarily small (i.e. the action may not be free on $\mathbb{C}^2 - \{0\}$). Then the quotient variety $X = \mathbb{C}^2/G$ and its projection $\mathbb{C}^2 \xrightarrow{\pi} X$ is equipped with a boundary divisor $B$ determined by the equality $K_{\mathbb{C}^2} = \pi^*(K_X + B)$ expressed as $B = \Sigma \frac{m_j - 1}{m_j} B_j$, where $B_j \subset X$ is the image of a one-dimensional linear subspace whose pointwise stabilizer subgroup $G_j \subset G$ is cyclic of order $m_j$. Furthermore, $G$ is small if and only if $B = 0$.  

\begin{example}
	Consider the abelian group generated by the matrices $\begin{bmatrix}
		1 & 0 \\ 0 & \epsilon_3
	\end{bmatrix}$ and $\begin{bmatrix} -1 & 0 \\ 0 & 1 \end{bmatrix}$, i.e. this is the abelian group $G \cong \mathbb{Z}_3 \times \mathbb{Z}_2$.
	
	There is a relation between canonical divisors: $K_{\mathbb{C}^2} = \pi^*(K_{\mathbb{C}^2/G} + \frac{1}{2}div(x^2) + \frac{2}{3}div(y^3)).$
\end{example}

\begin{theorem}[Proposition 5.20 of \cite{kollarmori}]
	The log pair $(X,B)$ is a log terminal singularity.
\end{theorem}

From this theorem, given a resolution of singularities $\tau: Y \rightarrow X$ and write $K_Y + \tau_*^{-1}(B) = \tau^*(K_X + B) + \Sigma_i a_i E_i$, where $E_i$ are the exceptional divisors and $a_i \in \mathbb{Q}$, then $a_i > -1$, for all $i$. Then, among all the resolutions $Y$ which satisfy $a_i \leq 0$ for all $i$, we define the maximal resolution of $(X,B)$:

\begin{definition}\label{maxdef}
	Let $(X,B)$ be a log terminal pair of a surface $X$ and a $\mathbb{Q}-$divisor $B$. We can assume the surfaces $Y$ and $Z$ are smooth. A resolution of singularities $f: Y \rightarrow X$ is a \textbf{maximal resolution} of $(X,B)$ if $K_Y + f_*^{-1}(B) = f^*(K_X + B) + \Sigma_i a_i E_i$, where $-1 < a_i \leq 0$, and for any proper birational morphism $g: Z \rightarrow Y$ that is not an isomorphism, we have $K_Z + h_*^{-1}(B) = h^*(K_X + B) + \Sigma_j b_j F_j$, $h = fg$ and for some $b_j > 0$.
\end{definition}

\begin{example} In the following, we all consider $X = \mathbb{C}^2/G$.
	\begin{enumerate}
		\item Consider the cyclic group $G = \frac{1}{5}(1,2)$. The minimal resolution is given by the dual graph in Figure \ref{minmax}.
		\begin{figure}
			\begin{tikzcd}[row sep=1em,column sep=1em]
				\circ_{(-\frac{2}{5},-3)} \arrow[r, dash] & \circ_{(-\frac{1}{5},-2)}
			\end{tikzcd}
			\caption{Dual graph of the exceptional divisors of the minimal resolution of $X$, where $(a_i,b_i)$ is the ordered pair whose $a_i$ is the coefficient of $E_i$ and $b_i$ is the self-intersection number.}\label{minmax}
		\end{figure}
		The dual graph after the blow-up over the intersection point of two exceptional divisors is given in Figure \ref{blwmin}.
		\begin{figure}
			\begin{tikzcd}[row sep=1em,column sep=1em]
				\circ_{(-\frac{2}{5},-4)} \arrow[r, dash] & \circ_{(\frac{2}{5},-1)} \arrow[r, dash] & \circ_{(-\frac{1}{5},-3)}
			\end{tikzcd}
			\caption{Dual graph of the exceptional divisors of the blow-up of minimal resolution of $X$ over the intersection point, where $(a_i,b_i)$ is again from Figure \ref{minmax}.}\label{blwmin}
		\end{figure}
		In this case, the maximal resolution of $X$ coincides with the minimal resolution.
		\item A slightly more complicated example is the Example 3.15 in \cite{kollarshep}, where $G = \frac{1}{19}(1,7)$. In this case, the maximal resolution of $X$ is not isomorphic to the minimal resolution.
        \item The first two examples are small groups. Now, we consider the smallest reflection group $G = \left\langle \begin{pmatrix}
	0 & 1\\
	1 & 0
\end{pmatrix} \right\rangle \cong \mathbb{Z}_2$. The quotient variety $X$ is smooth, so the minimal resolution is the identity morphism. The maximal resolution coincides with the minimal resolution as follows:
\begin{align*}
\mathbb{C}^2/G &= \operatorname{Spec}(\mathbb{C}[x,y]^G) = \operatorname{Spec}(\mathbb{C}[a:=x+y, b:=xy])\\
B &:= (x - y)^2 = a^2 - 4b, \mbox{ for any point $(a,b)$ on $B$,}\\
f &:= Blp_{(x+y,xy) = (a, a^2/4)} : Y \rightarrow X\\
K_Y + f_*^{-1}(B) &= f^*(K_X + B) + cE\\
c &= - K_Y \cdot E - f_*^{-1}(B) \cdot E\\
&= 1 - \frac{1}{2} = \frac{1}{2}
\end{align*}
    Thus, the maximal resolution of $(\mathbb{C}^2/G, B)$ is the identity map.
	\end{enumerate}
\end{example}

\begin{theorem}[expanded from \cite{kollarshep}, Lemma 3.13; generalized in higher-dimension cases in Theorem 17 of \cite{kawamata}; and Corollary 1.4.3 of \cite{bchm}]
	A quotient singularity $(X,B)$ of a surface has a unique maximal resolution (which we denote by $Y_{max}$). 
\end{theorem}

\section{Realizing Blow-ups as Moduli Spaces}
In this section, we prove Theorem 1.2 (or the conjecture in the case of dihedral groups) by embedding an affine open subset of each blow-up of $\mathbb{C}^2/D_{2n}$ to a crepant resolution of $\mathbb{C}^3/D_{2n}$.

Throughout the rest of this paper (unless explicitly mentioned), in $G := D_{2n}$ represented by $\left\langle \sigma = \begin{pmatrix}
	e^{2\pi i/n} & 0\\
	0 & e^{-2\pi i/n}
\end{pmatrix}, \tau = \begin{pmatrix}
	0 & 1\\
	1 & 0
\end{pmatrix} \right\rangle$ and $H := \langle \sigma \rangle \cong \mathbb{Z}_n$, we have the following commutative diagram:

\begin{tikzcd}
	\mathbb{C}^2  \arrow[d, "\tau_1"] & &\\
	\mathbb{C}^2/H \arrow[d, "\tau_2"] & X_1 \arrow[d, "\tau_2'"] \arrow[l, "f_1"]& Y \arrow[dl, "f_3"]\\
	\mathbb{C}^2/G & Y_1 \arrow[l, "f_2"] & (Y_1)_{max} \arrow[l, "h_{max}"] 
\end{tikzcd}

where the corresponding varieties and the morphisms are:
\begin{align*}
	X_1 &:= H\operatorname{-Hilb}(\mathbb{C}^2) = \mathbb{Z}_n\operatorname{-Hilb}(\mathbb{C}^2)\\
	Y_1 &:= X_1/(G/H) = X_1/(\mathbb{Z}_2)\\
	f_1 &: \mbox{ the minimal (crepant) resolution of $\mathbb{C}^2/H$}\\
	\tau_1 &: \mbox{ projection morphisms to $H$-orbits}\\
	\tau_2, \tau_2' &: \mbox{ projection morphisms to their $G/H \cong \mathbb{Z}_2$-orbits}\\
	f_2 &: \mbox{ the induced morphism from taking $\mathbb{Z}_2$-orbits}\\
	(Y_1)_{max} &:= \mbox{ the maximal resolution $h_{max}$ of $(Y_1, B')$, where $B'$ is defined by:}\\
	K_{X_1} &= \tau_2'^*(K_{Y_1} + B')\\
	f_3 &: \mbox{ the minimal resolution of $Y_1$}\\
	Y &:= (G/H)\operatorname{-Hilb}(X_1) = \mathbb{Z}_2\operatorname{-Hilb}(X_1)
\end{align*}

By the commutative diagram above, because $f_1$ is a birational map, $f_2$ is also a birational map. We can see this because the projection $\tau_2$ induces an inclusion between the ring of rational functions $(k(\mathbb{C}^2)^H)^{G/H} \hookrightarrow k(\mathbb{C}^2)^H$.

We denote the following exceptional divisors and refer to Figures \ref{config1} and \ref{config2} for the configuration:
\begin{equation}\label{excd}
\begin{aligned}
	(1) \mbox{ On $X_1$, $\tilde{E_i}$ whose projective coordinates are defined by:}\\ (x^i : y^{n-i}), (1 \leq i \leq n-1).\\
	(2) \mbox{ On $Y_1$, $E_i := \tau_2'(\tilde{E}_i)$ so $\tau_2'(\tilde{E}_{n-i}) = \tau_2'(\tilde{E}_i)$}\\
	\mbox{for all $1 \leq i \leq m$, where $m = \frac{n-1}{2}$ for odd $n$ and $m = n/2$ for even $n$}.
\end{aligned}
\end{equation}

We define the ramification divisors on $\mathbb{C}^2/G$ (i.e. the support of the discriminant divisor $\hat{B}$ defined by the equation $K_{\mathbb{C}^2/H} = \tau_2^*(K_{\mathbb{C}^2/G} + \hat{B})$) with their corresponding explicit equations as:
\begin{equation}\label{eqb1}
	\begin{aligned}
		\hat{B}_1: \langle (x^{n/2} + y^{n/2})^2 = 0 \rangle \\[1pt]
		\hat{B}_2: \langle (x^{n/2} - y^{n/2})^2 = 0 \rangle \\[1pt]
		\hat{B}_3: \langle (x^n - y^n)^2 = 0 \rangle
	\end{aligned}
\end{equation}
so that we can define their corresponding strict transformations for $i = 1,2,3$ as:
\begin{equation}\label{eqb2}
	\begin{aligned}
		B_i := (f_2)_*^{-1} (\hat{B}_i)\\
		\tilde{B_i} := (\tau_2')_*^{-1}(B_i).
	\end{aligned}
\end{equation}

\begin{figure}
	\centering
	\begin{tikzpicture}[scale=0.85]
		\draw[gray, thick] (0,0) -- (2.5,1);
		\draw[gray, thick] (2,1) -- (4,0);
		\foreach \Point in {(4.5,0), (4.75,0), (5,0), (5.25,0)}{
			\node at \Point {\textbullet};
		}
		\draw[gray, thick] (5.5,0) -- (7.5,1);
		\draw[gray, thick] (7,1) -- (9,0);
		\draw[red, thick] (7.25,1) -- (7.25,-0.25);
		\foreach \Point in {(9.5,1), (9.75,1), (10,1), (10.25,1)}{
			\node at \Point {\textbullet};
		}
		\draw[gray, thick] (10.5,1) -- (13,0) ;
		\draw[gray, thick] (12.5,0) -- (14,1);
		\node[text width = 0.1cm, fill=white] at (1.5,-0.5) 
		{$\tilde{E}_1:(x:y^{n-1})$};
		\node[text width = 0.1cm, fill=white] at (3,-0.5) 
		{$\tilde{E}_2:(x^2:y^{n-2})$};
		\node[text width = 0.1cm, fill=white] at (5.25,-1) 
		{$\tilde{E}_{(n-1)/2}:(x^{(n-1)/2}:y^{(n+1)/2})$};
		\node[text width = 0.1cm, fill=white] at (7.75,-1) 
		{$\tilde{E}_{(n+1)/2}:(x^{(n+1)/2}:y^{(n-1)/2})$};
		\node[text width = 0.1cm, fill=white] at (7.125,-0.5) 
		{$\tilde{B}_3$};
		\node[text width = 0.1cm, fill=white] at (11,-0.5) 
		{$\tilde{E}_{n-2}:(x^{n-2}:y^2)$};
		\node[text width = 0.1cm, fill=white] at (13.1,-0.65) 
		{$\tilde{E}_{n-1}:(x^{n-1}:y)$};
	\end{tikzpicture}
	
	\begin{center}
		$\downarrow{\tau_2'}$
	\end{center}
	
	\begin{tikzpicture}[scale=0.85]
		\draw[gray, thick] (0,0) -- (3,1);
		\draw[gray, thick] (2.5,1) -- (5,0);
		\foreach \Point in {(5.5,0), (5.75,0), (6,0), (6.25,0)}{
			\node at \Point {\textbullet};
		}
		\draw[gray, thick] (6.5,0) -- (14,1);
		\draw [red] plot [smooth, tension=1.125] coordinates { (9,-0.6) (11.353,0.647) (13,-0.5) };
		\node[text width = 0.25cm, fill=white] at (1.5,0) 
		{$E_1$};
		\node[text width = 0.25cm, fill=white] at (3,0.25) 
		{$E_2$};
		\node[text width = 0.25cm, fill=white] at (7,-0.5) 
		{$E_{(n-1)/2}$};
		\node[text width = 0.25cm, fill=white] at (11,0) 
		{$\tau_2'(\tilde{B}_3)$};
	\end{tikzpicture}
	\caption{Configuration of Exceptional Divisors and Boundary Divisors on $X_1$ and $Y_1$ for Odd $n$ Case}\label{config1}
\end{figure}
\begin{figure}
	\begin{tikzpicture}[scale=0.85]
		\draw[gray, thick] (0,0) -- (3,1);
		\draw[gray, thick] (2.5,1) -- (5,0);
		\foreach \Point in {(5.5,0), (5.75,0), (6,0), (6.25,0)}{
			\node at \Point {\textbullet};
		}
		\draw[gray, thick] (6.5,0) -- (9,1);
		\foreach \Point in {(9.5,1), (9.75,1), (10,1), (10.25,1)}{
			\node at \Point {\textbullet};
		}
		\draw[gray, thick] (10.5,1) -- (13,0) ;
		\draw[gray, thick] (12.5,0) -- (14,1);
		\draw[red, thick] (7.5,1) -- (8.5,-1);
		\draw[red, thick] (8.5,1) -- (9.5,-1);
		\node[text width = 0.1cm, fill=white] at (1.5,-0.5) 
		{$\tilde{E}_1:(x:y^{n-1})$};
		\node[text width = 0.1cm, fill=white] at (3,-0.5) 
		{$\tilde{E}_2:(x^2:y^{n-2})$};
		\node[text width = 0.1cm, fill=white] at (6.5,-1) 
		{$\tilde{E}_{n/2}:(x^{n/2}:y^{n/2})$};
		\node[text width = 0.1cm, fill=white] at (8,-1) 
		{$\tilde{B}_1$};
		\node[text width = 0.1cm, fill=white] at (9.5,-1) 
		{$\tilde{B}_2$};
		\node[text width = 0.1cm, fill=white] at (11,-0.5) 
		{$\tilde{E}_{n-2}:(x^{n-2}:y^2)$};
		\node[text width = 0.1cm, fill=white] at (13.1,-0.65) 
		{$\tilde{E}_{n-1}:(x^{n-1}:y)$};
	\end{tikzpicture}
	
	\begin{center}
		$\downarrow{\tau_2'}$
	\end{center}
	
	\begin{tikzpicture}[scale=0.85]
		\draw[gray, thick] (0,0) -- (3,1);
		\draw[gray, thick] (2.5,1) -- (5,0);
		\foreach \Point in {(5.5,0), (5.75,0), (6,0), (6.25,0)}{
			\node at \Point {\textbullet};
		}
		\draw[gray, thick] (6.5,0) -- (14,1);
		\draw[red, thick] (8,1) -- (9,-0.35);
		\draw[red, thick] (11,1) -- (12,0);
		\node[text width = 0.25cm, fill=white] at (1.5,0) 
		{$E_1$};
		\node[text width = 0.25cm, fill=white] at (3,0.25) 
		{$E_2$};
		\node[text width = 0.25cm, fill=white] at (7,-0.5) 
		{$E_{n/2}$};
		\node[text width = 0.25cm, fill=white] at (9.5,-0.5) 
		{$\tau_2'(\tilde{B}_1)$};
		\node[text width = 0.25cm, fill=white] at (12,0) 
		{$\tau_2'(\tilde{B}_2)$};
	\end{tikzpicture}
	\caption{Configuration of Exceptional Divisors and Boundary Divisors on $X_1$ and $Y_1$ for Even $n$ Case}\label{config2}
\end{figure}

Using the notations in \cite{nollasekiya}, we define $f_1 := x^{2m+1} + y^{2m+1}$ and $f_2 := x^{2m+1} - y^{2m+1}$ in the odd $n$ case; and $f_1 := x^m + y^m$ and $f_2 := x^m - y^m$ in the even $n$ case. We also note here that $\mathbb{C}[x,y]^{D_{2n}} = \mathbb{C}[xy, x^n + y^n]$. 

\begin{figure}
	\begin{tikzcd}[row sep=1em,column sep=1em]
		\circ_{E_1: \rho_1} \arrow[r, dash] & \circ_{E_2: \rho_2} \arrow[r, dash]\arrow[l, dash] & \cdots \arrow[r, dash]\arrow[l, dash] & \circ_{E_{(n-1)/2}: \rho_{\frac{n-1}{2}}} \arrow[r, red, dash]\arrow[l, dash] & \circ_{B: \rho_{\frac{n-1}{2}}} \\
	\end{tikzcd} \begin{tikzcd}[row sep=1em,column sep=1em]
		& & & & & \circ_{B_1: \rho'_{n/2}}  \\
		& \circ_{E_1: \rho_1} \arrow[r, dash] & \circ_{E_2: \rho_2} \arrow[r, dash]\arrow[l, dash] & \cdots \arrow[r, dash]\arrow[l, dash] & \circ_{E_{n/2}:\rho_{n/2} \oplus \rho_{n/2}'} \arrow[red,ur, dash]\arrow[l, dash]\\
		& & & & & \circ_{B_2: \rho_{n/2}} \arrow[red,ul, dash]
	\end{tikzcd}
	\caption{Dual graph of the exceptional divisors and boundary divisors of $f_2$}
\end{figure}

We prepare some propositions.

\begin{proposition}\label{smth}
	The surface $Y_1$ is smooth. Hence, $f_2$ is a resolution of $(\mathbb{C}^2/G, \hat{B})$.
\end{proposition}

\begin{proof}
	We compute $(X_1)^{\mathbb{Z}_2}$, the fixed locus of the $\mathbb{Z}_2$-action on $X_1$ by taking a closed subscheme $V$ to $g \cdot V$, where $g$ is an element of $\mathbb{Z}_2$.
	
	The points of $X_1$ are $G$-invariant $0$-dimensional subscheme of $\mathbb{C}^2$ (whose space of global sections is isomorphic to the regular representation, i.e. $H^0(O_Z) \cong \mathbb{C}[G]$), and so it can be realized as an ideal defining the aforementioned closed subscheme of $\mathbb{C}^2$. Referring to Thm. 2.2 of \cite{itonak1}; and Remark 9.7, Lemma 12.2, and Theorem 12.3 of \cite{itonak2}, we can identify the points on the exceptional divisors of $X_1$ as $I_i(a_i : b_i) = \langle a_ix^i - b_iy^{n-i}, x^{i+1}, xy, y^{n+1-i} \rangle$, where $1 \leq i \leq n$ and $(a_i:b_i) \in \mathbb{P}^1$; or equivalently, using the open affine covers of $X_1 = \bigcup_{i=1}^n U_i := \bigcup_{i=1}^n \operatorname{Spec}\left(\mathbb{C}\left[ \frac{x^i}{y^{n-i}}, \frac{y^{n+1-i}}{x^{i-1}} \right]\right)$, the points $\left( \frac{x^i}{y^{n-i}}, \frac{y^{n+1-i}}{x^{i-1}} \right) = (0, b)$ and $\left( \frac{x^i}{y^{n-i}}, \frac{y^{n+1-i}}{x^{i-1}} \right) = (a, 0)$ on the exceptional divisor correspond to $I_{i-1}(b:1)$ and $I_i(1:a)$, respectively.
	
	Furthermore, the $\mathbb{Z}_2$-action sends $I_i(a_i : b_i) = \langle a_ix^i - b_iy^{n-i}, x^{i+1}, xy, y^{n+1-i} \rangle$ to $I_{n-i}(b_i :a_i) = \langle a_iy^i - b_ix^{n-i}, x^{n-i+1}, xy, y^{i+1} \rangle$, so that the fixed points on the exceptional divisors of $X_1$ are:
	
	When $n$ is odd, $I_{(n-1)/2}(0:1) = I_{(n+1)/2}(1:0)$.
	
	When $n$ is even, there are two fixed points, $I_{n/2}(1:1)$ or $I_{n/2}(-1:1)$.
	
	From the commutative diagram below, we can check smoothness of $Y_1$ by verifying that the $\mathbb{Z}_2$ acts on $X_1$ as a pseudoreflection. This amounts in showing that the fixed points under the $\mathbb{Z}_2$-action are on $\overline{((f_1)_*^{-1} \circ (\tau_2)_*^{-1}) (\hat{B}_{i,0})}$.
	
	\begin{tikzcd}
		(\tau_2)_*^{-1}(\hat{B}_{i,0}) \subset \mathbb{C}^2/H - \{(0,0)\} \arrow[d, "\tau_2"] & X_1 - \Sigma(\tilde{E}_i) \arrow[d, "\tau_2'"] \arrow[l, "f_1"]\\
		\hat{B}_i - \{(0,0)\} =: \hat{B}_{i,0} \subset \mathbb{C}^2/G - \{(0,0)\} & Y_1 - \Sigma(E_i) \arrow[l, "f_2"]
	\end{tikzcd}
	
	First, we define the open affine covering of $X_1 = \bigcup_{i=1}^n U_i$ as:
	$$X_1 = \bigcup_{i=1}^n U_i := \bigcup_{i=1}^n \operatorname{Spec}\left(\mathbb{C}\left[ \frac{x^i}{y^{n-i}}, \frac{y^{n+1-i}}{x^{i-1}} \right]\right).$$
	
	By the definition of $\hat{B}_i$ from (\ref{eqb1}), we compute the strict transform $((f_1)_*^{-1} \circ (\tau_2)_*^{-1}) (\hat{B}_i)$ on each of the affine open sets covering $X_1$.
	
	On $\operatorname{Spec}\left(\mathbb{C}\left[ \frac{x^i}{y^{n-i}}, \frac{y^{n+1-i}}{x^{i-1}}   \right]\right)$,
	\begin{align*}
	(x^n - y^n)^2 &= \left(\frac{x^i}{y^{n-i}}\right)^{2(n+1-i)} \left(\frac{y^{n+1-i}}{x^{i-1}}\right)^{2(n-i)} \\
	& - 2 \left(\frac{x^i}{y^{n-i}}\right)^n \left(\frac{y^{n+1-i}}{x^{i-1}}\right)^n + \left(\frac{x^i}{y^{n-i}}\right)^{2(i-1)} \left(\frac{y^{n+1-i}}{x^{i-1}}\right)^{2i}
	\end{align*}
	For $\hat{B_1}$:
	$$0 = x^n + 2(xy)^{n/2} + y^n = \left(\frac{x^{n/2}}{y^{n/2}}\right)^{\frac{n-2}{2}} \left(\frac{y^{(n+2)/2}}{x^{(n-2)/2}}\right)^{n/2} \left[ \left(\frac{x^{n/2}}{y^{n/2}}\right) + 1 \right]^2$$
	so that the strict transform is the line $\frac{x^{n/2}}{y^{n/2}} = -1$ on the open set $U_{n/2}$. The coordinate $(\frac{x^{n/2}}{y^{n/2}}, \frac{y^{(n+2)/2}}{x^{(n-2)/2}}) = (-1,0)$ corresponds to the point on the $G$-Hilbert scheme $I_{n/2}(1:-1)$. This works similarly for $U_{(n/2) + 1}$ and $U_{(n+1)/2}$ (for odd $n$). The same argument works for $\hat{B_2}$ in which we obtain $I_{n/2}(1:1)$; and for $\hat{B_3}$ in which we obtain $I_{(n-1)/2}(0:1) = I_{(n+1)/2}(1:0)$. This completes the description of the strict transform of the boundary divisors.
	
	To show further that the closure $\overline{((f_1)_*^{-1} \circ (\tau_2)_*^{-1}) (\hat{B}_i)}$ does not exist on other open sets other than $U_{n/2}, U_{(n/2) + 1}, U_{(n+1)/2}$, we again notice that WLOG:
	
	\begin{align*}
		(x^n - y^n)^2 &= \left(\frac{x^i}{y^{n-i}}\right)^{2(i-1)} \left(\frac{y^{n+1-i}}{x^{i-1}}\right)^{2i} \cdot \\
		& \left[ \left(\frac{x^i}{y^{n-i}}\right)^{2(n+2-2i)} \left(\frac{y^{n+1-i}}{x^{i-1}}\right)^{2(n-2i)} - 2 \left(\frac{x^i}{y^{n-i}}\right)^{n + 2 - 2i} \left(\frac{y^{n+1-i}}{x^{i-1}}\right)^{n-2i} + 1 \right]
	\end{align*}
	
	This implies that plugging any of the coordinates of $U_i$ to zero does not lie on the strict transform defined by: $$0 = \left(\frac{x^i}{y^{n-i}}\right)^{2(n+2-2i)} \left(\frac{y^{n+1-i}}{x^{i-1}}\right)^{2(n-2i)} - 2 \left(\frac{x^i}{y^{n-i}}\right)^{n + 2 - 2i} \left(\frac{y^{n+1-i}}{x^{i-1}}\right)^{n-2i} + 1.$$ This completes the description of the strict transform of the boundary divisors.
	
	These all imply that $dim((X_1)^{\mathbb{Z}_2}) = 1$, or equivalently, $\mathbb{Z}_2$ acts as a pseudo-reflection on $X_1$, which implies that the boundary divisor $B'$ on $Y_1$ determined by the equation $K_{X_1} = \tau_2'^*(K_{Y_1} + B')$ is smooth and so $Y_1$ is smooth. 
\end{proof}

\begin{remark}
	A more general statement for the smoothness of $\tau_2'$ is as follows:
	
	\begin{corollary}
		For $H \subset \operatorname{SL}(2), Y_1 := H\operatorname{-Hilb}(\mathbb{C}^2)$, so that $G/H$ is cyclic, then $Y_1/(G/H)$ is smooth iff $G/H$ is a (cyclic) complex reflection group (or equivalently, if $G/H$ has a local linear action on $Y_1$ by pesudoreflections).
	\end{corollary}
	
	The dihedral group is a special case of this.
\end{remark}

\begin{proposition}\label{crep41}
	The morphism $f_2: Y_1 \rightarrow (\mathbb{C}^2/G, \hat{B})$ is a crepant resolution.
\end{proposition}

\begin{proof}
	From the fact that $f_1$ is a crepant resolution:
	\begin{align*}
		K_{X_1} &= f_1^*(K_{\mathbb{C}^2/H})\\
		K_{Y_1} + (f_2)_*^{-1}(\hat{B}) &= f_2^*(K_{\mathbb{C}^2/G} + \hat{B}) + \Sigma_j a_j F_j\\
		K_{X_1} &= (\tau_2')^*(K_{Y_1} + B')
	\end{align*}
	We need to show that $(f_2)_*^{-1}(\hat{B}) = B' = B_{Y_1}$ by showing that $a_j = 0$ for all $j$.
	\begin{align*}
		K_{\mathbb{C}^2/H} &= \tau_2^*(K_{\mathbb{C}^2/G} + \hat{B})\\
		K_{X_1} = f_1^*(K_{\mathbb{C}^2/H}) &= f_1^*(\tau_2^*(K_{\mathbb{C}^2/G} + \hat{B})) \\
		&= (\tau_2 \circ f_1)^*(K_{\mathbb{C}^2/G} + \hat{B})\\
		&= (f_2 \circ \tau_2')^*(K_{\mathbb{C}^2/G} + \hat{B})\\
		&= \tau_2'^*(K_{Y_1} + (f_2)_*^{-1}(\hat{B}) - \Sigma_j a_j F_j),\\
		&\mbox{ (where $F_j$ are exceptional divisors of $f_2$)}\\
		&= \tau_2'^*(K_{Y_1} + (f_2)_*^{-1}(\hat{B})) - \Sigma_j a_j \tau_2'^*(F_j)
	\end{align*}
	This implies that $\tau_2'^*(F_j)$ are exceptional divisors for $f_1$ which forces the discrepancies to be zero.
\end{proof}

This proposition implies that $Y_1$ is also a crepant resolution of $(\mathbb{C}^2/G, \hat{B})$. Also, we recall the notion of the minimal embedded resolution of $(\mathbb{C}^2/G, \hat{B})$.

\begin{proposition}[Proposition 3.8 (Ch. V) of \cite{hartshorne1}]
	Let $C_0$ be an irreducible curve in the surface $X_0$. Then there exists a finite sequence of monoidal transformations (with suitable centers) $X_n \rightarrow X_{n-1} \rightarrow ... \rightarrow X_1 \rightarrow X_0$ such that the strict transform $C_n$ of $C_0$ on $X_n$ is nonsingular.
\end{proposition}

The minimum $n$ that satisfies this proposition gives the minimal embedded resolution $X_n$ of $(X_0, C_0)$. This is minimal in the sense that if $Y$ is a smooth surface and dominates $X_0$, then it dominates $X_n$ as well. For instance, a smooth surface $Y$ with normal crossings dominates the minimal embedded resolution. A more detailed description is given in Theorem 3.9 (Ch. V) of \cite{hartshorne1}.

\begin{theorem}\label{maxemb}
	The maximal resolution $Y_{max}$ of $(\mathbb{C}^2/G, \hat{B})$ is isomorphic to the quotient variety $Y_1$. It is also the minimal embedded resolution of $(\mathbb{C}^2/G, \hat{B})$. Furthermore, the iterated Hilbert scheme $Y$ is also isomorphic to the maximal resolution $Y_{max}$.
\end{theorem}

\begin{proof}
	First, from the proof of Theorem 1 in \cite{kawamata}, the maximal resolution of $(Y_1, B')$ defined by $h_{max}$ is also the maximal resolution of $(\mathbb{C}^2/G, \hat{B})$ defined by $h_{max} \circ f_2$. 
	
	Because $Y_1$ has at worst cyclic quotient singularities, and $(Y_1, B_{Y_1})$ has the smoothness property for both the variety and the boundary divisor, the minimal resolution of $(Y_1, (f_2)_*^{-1}(\hat{B}))$, i.e. $f_3: Y \rightarrow (Y_1, (f_2)_*^{-1}(\hat{B}))$ is crepant, and more strongly, $f_3 = id$, which implies that the maximal resolution of $(\mathbb{C}^2/G, \hat{B})$ is $Y_1$.
\end{proof}

\begin{remark}\label{oddboundary}
	An explicit way to do this is to consider the affine open covers of $Y_1$ via the open affine covers of $X_1$. We show this for the odd $n$ case, since the argument for the even case is similar.
	
	Because $f_2$ is a crepant resolution, it remains to compute the discrepancy of the blow-up $h: Z := Blp(Y_1) \rightarrow Y_1$, and we divide it into two cases depending on where the center of the blow-up is. The more interesting case is where the center of $h$ is on $f_{2*}^{-1}(\hat{B})$ (this can be realized also as the boundary divisor for the morphism $\tau_2'$):
	
	In the odd $n$ case, in $\mathbb{Z}_n\operatorname{-Hilb}(\mathbb{C}^2)$, the open set $\operatorname{Spec}\left(\mathbb{C}\left[\frac{x^{(n+1)/2}}{y^{(n-1)/2}}, \frac{y^{(n+1)/2}}{x^{(n-1)/2}}\right]\right)$ which covers the invariant locus under the $\mathbb{Z}_2$-action is $\mathbb{Z}_2$-invariant. Thus, we consider the open set $\operatorname{Spec}\left(\mathbb{C}\left[\frac{x^{(n+1)/2}}{y^{(n-1)/2}}, \frac{y^{(n+1)/2}}{x^{(n-1)/2}}\right]\right)^{\mathbb{Z}_2} = \operatorname{Spec}\left(\mathbb{C}\left[xy, \frac{f_1}{(xy)^m}\right]\right)$. The boundary locus in $\mathbb{Z}_n\operatorname{-Hilb}(\mathbb{C}^2)$ is $\left(\frac{x^{(n+1)/2}}{y^{(n-1)/2}} - \frac{y^{(n+1)/2}}{x^{(n-1)/2}} \right)^2 = 0$, which translates to $\left(\frac{f_1}{(xy)^m}\right)^2 - 4xy = 0$ on the invariant open set.
	
	For any point $\left(xy, \frac{f_1}{(xy)^m}\right) = \left(\frac{1}{4}a^2, a\right)$ on $\tilde{B}_m$, performing the coordinate change, we obtain the new equation: $\left(\frac{f_1}{(xy)^m}\right)^2 + 2a\frac{f_1}{(xy)^m} + a^2 = \left(\frac{f_1}{(xy)^m} + a\right)^2 = 4(xy + \frac{1}{4}a^2) = 4xy + a^2$.
	
	On $\operatorname{Spec}\left(\mathbb{C}\left[\frac{(xy)^{m+1}}{f_1}, \frac{f_1}{(xy)^m}\right]\right)$, the defining equation transforms to $\frac{f_1}{(xy)^m} - 2a = 4\frac{(xy)^{m+1}}{f_1}$. The exceptional divisor defines the equation $\frac{f_1}{(xy)^m} = 0$.
	
	Thus, the intersection number of $h_*^{-1}(B')$ with the exceptional divisor is $1/2$. Using the relation between canonical divisors, the discrepancy $a_{m+1} = 1/2$. For the even $n$ case, this reduces to a blow-up along lines which is treated similarly.
	
	Thus, $Y_1$ is the maximal resolution of $(\mathbb{C}^2/D_{2n}, \hat{B})$. Furthermore, because $f_3$ is a crepant resolution of $Y_1$, $f_3$ must be an isomorphism. Hence, $Y \cong (\mathbb{C}^2/G, \hat{B})_{max} \cong Y_1$.
\end{remark}

This particular assertion tells us that the maximal resolution can be realized as a moduli space of $G$-constellations which will help in our computations later.

Once again, we refer to (\ref{excd}) for the definition of the exceptional divisors for the next lemma:
\begin{lemma}\label{normal}
	For an exceptional divisor $\tilde{E}$ (resp. $E$) of $X_1$ (resp. $Y_1$), we know that the normal bundles $\mathcal{N}_{\tilde{E}/X_1}$ are of degree $-2$, or equivalently, $\mathcal{N}_{\tilde{E}/X_1} \cong O_{\tilde{E}}(-2)$. Then:
	\[\mathcal{N}_{E/Y_1} \cong \begin{cases}
		O_E(-1) & \mbox{ if $E = E_m$}\\
		O_E(-2) & \mbox{ if $E \neq E_m$}\\
	\end{cases} \]
\end{lemma}

\begin{proof}
	The first statement for $X_1$ is well-known since it is the (minimal) crepant resolution of the quotient singularity $\mathbb{C}^2/\mathbb{Z}_n$.
	
	We can compute the self-intersection number $E^2$ via the adjunction formula and given our computations in Proposition \ref{smth} regarding the fixed points of $X_1$ under the $\mathbb{Z}_2$-action.
	
	Because $f_2$ is crepant, we have $K_{Y_1} + (f_2)_*^{-1}(\hat{B}) = (f_2)^*(K_{\mathbb{C}^2/G} + \hat{B})$, so that $K_{Y_1} \cdot E_m = -1$ and $K_{Y_1} \cdot E = 0$ for $E \neq E_m$.
\end{proof}

\begin{corollary}\label{resev}
	The only resolutions dominated by the maximal resolution of $(\mathbb{C}^2/G, \hat{B})$ are essentially the blow-ups from $(\mathbb{C}^2/G, \hat{B})$ with center the singular point of the (strict transforms of the) boundary divisor $\hat{B}$.
\end{corollary}

\begin{proof}
	Using the same argument as in Lemma \ref{normal}, after the blow-down of the $(-1)$-curve on $Y_1$ and so on, we obtain the result.
\end{proof}

The next lemma provides an isomorphism between the minimal resolution of the variety and the quotient variety.

\begin{lemma}
	For a complex reflection group $G \subset \operatorname{GL}(n,\mathbb{C})$, there is an isomorphism between the $G$-Hilbert scheme and the quotient variety. In symbols:
	$$G\operatorname{-Hilb}(\mathbb{C}^n) \cong \mathbb{C}^n/G.$$
\end{lemma}

\begin{proof}
	We consider the moduli functor for $G$-clusters:
	$$h: S \mapsto \{\mbox{flat families of $G$-clusters parametrized by $S$} \} / \equiv$$
	for a locally Noetherian scheme $S$ over $\mathbb{C}$ where $E_S \equiv F_S$ if and only if there is an $L$ in $\operatorname{Pic}(S)$ such that $E_S \cong F_S \otimes L$.
	
	The $G$-Hilbert scheme $G\operatorname{-Hilb}(\mathbb{C}^n)$ represents the functor $h$. Thus:
	$h(S) \cong \operatorname{Hom}_{Sch}(S, G\operatorname{-Hilb}(\mathbb{C}^n))$. We wish to show that $h(S) \cong \operatorname{Hom}_{Sch}(S, \mathbb{C}^n/G)$.
	
	We construct the map first from $h(S)$ to $\operatorname{Hom}_{Sch}(S, \mathbb{C}^n/G)$.
	
	(1) From $\upsilon: \operatorname{Hom}_{Sch}(S, \mathbb{C}^n/G)$ to $h(S)$.
	
	Given $\gamma_S \in \operatorname{Hom}_{Sch}(S, \mathbb{C}^n/G)$, consider the fiber product diagram:
	
	\begin{tikzcd}
		S \times_{\mathbb{C}^n/G} \mathbb{C}^n \arrow[d, "p_1"]\arrow[r, "p_2"] & \mathbb{C}^n \arrow[d, "p"]\\
		S \arrow[r, "\gamma_S"] & \mathbb{C}^n/G
	\end{tikzcd}
	
	It remains to show that every fiber of $p$ is a $G$-cluster. This implies that the fiber product $S \times_{\mathbb{C}^n/G} \mathbb{C}^n$ is a flat family of $G$-clusters.
	
	By the Chevalley-Shephard-Todd theorem, the morphism $p$ is flat. Then by the decomposition of $p_*(O_{\mathbb{C}^n/G}) = \bigoplus_{\rho \in \operatorname{Irr}(G)} M_{\rho} \otimes \rho$ over representations of a finite group $G$ with characteristic $0$, where $M_{\rho} = (p_*(O_{\mathbb{C}^n/G}) \otimes \rho^{\lor})^G$ is a finitely generated $O_{\mathbb{C}^n/G}$-module, each of the modules $M_{\rho}$ is locally free.
	
	Over the free locus on $\mathbb{C}^n/G$, the fiber consists of a $G$-cluster. Thus, considering a family of representations of a finite group $G$, the fiber over the non-free locus is also a $G$-cluster.
	
	(2) From $\Lambda: h(S)$ to $\operatorname{Hom}_{Sch}(S, \mathbb{C}^n/G)$.
	
	Given a flat family $\mathcal{Z}$ of $G$-clusters over a scheme $S$, which is a subscheme of $S \times \mathbb{C}^n$, we wish to construct a scheme morphism $\delta_{\mathcal{Z}}: S \rightarrow \mathbb{C}^n/G$. We consider first the following diagram:
	
	\begin{tikzcd}
		\mathcal{Z} \arrow[d, "p_1"]\arrow[r, "p_2"] & \mathbb{C}^n \arrow[d, "p"]\\
		S \arrow[r, dashed, "\delta_{\mathcal{Z}}"] & \mathbb{C}^n/G
	\end{tikzcd}
	
	Taking note that the action of $G$ on $S$ is trivial, so that again, there is a decomposition of $(p_1)_*(O_{\mathcal{Z}}) = \bigoplus_{\rho \in \operatorname{Irr}(G)} S_{\rho} \otimes \rho$.
	
	Taking the $G$-invariant sections gives $[(p_1)_*(O_{\mathcal{Z}})]^G = S_{\rho_0}$, which is of rank $1$, which is generated by the non-vanishing global section $1$. This implies that $S_{\rho_0} = O_S$ and $\mathcal{Z}/G = S$.
	
	Thus, the morphism $\delta_{\mathcal{Z}}: \mathcal{Z}/G = S \rightarrow \mathbb{C}^n/G$ is induced by the morphism $p_2: \mathcal{Z} \rightarrow \mathbb{C}^n$.
	
	(3) Now that we have constructed the maps, we want to show that the maps $\upsilon$ and $\Lambda$ induce a bijection between sets. First, we show that $\Lambda \circ \upsilon = 1_{\operatorname{Hom}_{Sch}(S, \mathbb{C}^n/G)}$.
	
	By the construction of the map $\Lambda$, $p_2$ induces the map $(p_2)/G: S = (S \times_{\mathbb{C}^n/G} \mathbb{C}^n/G) = (S \times_{\mathbb{C}^n/G} \mathbb{C}^n)/G \rightarrow \mathbb{C}^n/G$. From the construction, $S$ is a categorical quotient $p_1$. Thus, by the universality property of the categorical quotient applied to the morphism $p \circ p_2$, $\gamma_S = (p_2)/G$.
	
	(4) We now show that $\upsilon \circ \Lambda = 1_{h(S)}$. This amounts to show that $\mathcal{Z} = S \times_{\mathbb{C}^n/G} \mathbb{C}^n$.
	
	Consider the inclusion morphism $i: \mathcal{Z} \hookrightarrow S \times_{\mathbb{C}^n/G} \mathbb{C}^n$, which is a closed immersion (via the closed immersion $S \times_{\mathbb{C}^n/G} \mathbb{C}^n \hookrightarrow S \times_{\mathbb{C}} \mathbb{C}^n$), induced by the universal property of the fiber product diagram. Then we have the exact sequence:
	$$0 \rightarrow \mathcal{I} \rightarrow O_{S \times_{\mathbb{C}^n/G} \mathbb{C}^n} \rightarrow i_*(O_{\mathcal{Z}}) \rightarrow 0$$
	where $\mathcal{I}$ is the kernel of $O_{S \times_{\mathbb{C}^n/G} \mathbb{C}^n} \rightarrow i_*(O_{\mathcal{Z}})$.
	
	Because $p_1$ is finite, the pushforward functor $(p_1)_*$ is exact:
	$$0 \rightarrow (p_1)_*\mathcal{I} \rightarrow (p_1)_*(O_{S \times_{\mathbb{C}^n/G} \mathbb{C}^n}) \rightarrow (p_1)_*(i_*(O_{\mathcal{Z}})) = (p_1)_*(O_{\mathcal{Z}}) \rightarrow 0$$
	Because both $S \times_{\mathbb{C}^n/G} \mathbb{C}^n$ and $\mathcal{Z}$ are flat families of $G$-clusters over $S$, the sheaf $(p_1)_*(O_{\mathcal{Z}})$ is flat and every fiber of both $(p_1)_*(O_{S \times_{\mathbb{C}^n/G} \mathbb{C}^n})$ and $(p_1)_*(O_{\mathcal{Z}})$ are $G$-clusters. Taking the fibers over $s \in S$ in the exact sequence above implies that
	
	$[(p_1)_*(O_{S \times_{\mathbb{C}^n/G} \mathbb{C}^n})](s) := (p_1)_*(O_{S \times_{\mathbb{C}^n/G} \mathbb{C}^n})_s \otimes (O_{S,s}/\mathfrak{m}_{S,s})$ and 
	
	$[(p_1)_*(i_*(O_{\mathcal{Z}})) = (p_1)_*(O_{\mathcal{Z}})](s)$ have the same dimension as vector spaces over $\mathbb{C}$. Thus, they are isomorphic as vector spaces. This leaves the fiber $[(p_1)_*\mathcal{I}](s) = 0$ for all $s \in S$. By Nakayama Lemma for local rings applied to the coherent sheaf $(p_1)_*\mathcal{I}$, for $\mathcal{I}$ is coherent and $p_1$ is finite, this implies that the stalk $(p_1)_* (\mathcal{I})_s = 0$, which implies that $(p_1)_*(\mathcal{I}) = 0$.
	
	Again, because $p_1$ is finite, then the natural map $0 = (p_1)^*(p_1)_*(\mathcal{I}) \rightarrow \mathcal{I}$ is surjective, which implies that $\mathcal{I} = 0$. This implies now that $O_{S \times_{\mathbb{C}^n/G} \mathbb{C}^n} \cong i_*(O_{\mathcal{Z}})$, and so $\mathcal{Z} = S \times_{\mathbb{C}^n/G} \mathbb{C}^n$.
\end{proof}

\begin{remark}
	The $G$-cluster that corresponds to the points on the boundary divisor of $\mathbb{C}^n/G$ is a non-reduced closed subscheme.
\end{remark}

\begin{theorem}\label{maxmain}
	Given a log pair of the quotient variety $(\mathbb{C}^2/D_{2n}, \hat{B})$ determined by the projection morphism $\pi: \mathbb{C}^2 \rightarrow \mathbb{C}^2/D_{2n}$ via the relation $K_{\mathbb{C}^2} = \pi^*(K_{\mathbb{C}^2/D_{2n}} + \hat{B})$ from the previous section, then the following hold for the blow-ups of the quotient variety $\mathbb{C}^2/D_{2n}$:
	\begin{enumerate}
		\item The maximal resolution for $(\mathbb{C}^2/D_{2n}, \hat{B})$ is obtained after $m = \frac{n-1}{2}$ (for odd case) and $m = \frac{n}{2}$ (for even case) blowing ups with singular points of the boundary divisors as the center which satisfies the inequality in Definition \ref{maxdef}.
		\item For each resolution $\tilde{Y} \rightarrow \mathbb{C}^2/D_{2n}$ dominated by the maximal resolution $Y_{max}$, there is a generic $\theta$ in the parameter space of $G$-constellations such that $\tilde{Y} \cong \mathcal{M}_{\theta}$.  
	\end{enumerate}
\end{theorem}

\begin{proof}
	(1) Because $X_1$ has $n-1$ exceptional divisors, $Y_1$ has $\frac{n-1}{2}$ (for $n$ odd) or $\frac{n}{2}$ (for even $n$) exceptional divisors. The surface $Y_1$ being the maximal resolution follows from Theorem \ref{maxemb}. And the contraction of exceptional divisors yielding a smooth resolution follows from Corollary \ref{resev}.
	
	(2) Because $(Y_1)_{max} \cong Y_1 \cong Y$, we can look instead at the iterated Hilbert scheme and naturally embed $Y$ into $\mathbb{Z}_2\operatorname{-Hilb}(\mathbb{Z}_n\operatorname{-Hilb}(\mathbb{C}^3))$.
	
	We embed the group $D_{2n} \subset \operatorname{GL}(2) \hookrightarrow \operatorname{SL}(3)$ via taking the determinant so that $\operatorname{SL}(3) \supset D_{2n} =\left\langle\begin{bmatrix}
		0 & 1 & 0\\
		1 & 0 & 0\\
		0 & 0 & -1
	\end{bmatrix}, \sigma := \begin{bmatrix}
		\epsilon & 0 & 0\\
		0 & \epsilon^{-1} & 0\\
		0 & 0 & 1
	\end{bmatrix} \right\rangle$, so that it induces the group action of $D_{2n}$ on $\mathbb{C}^3$ by matrix multiplication as well). This iterated Hilbert scheme $\mathbb{Z}_2\operatorname{-Hilb}(\mathbb{Z}_n\operatorname{-Hilb}(\mathbb{C}^3))$ is identified with $X_{0...(m-1)}$ via Theorems 5.1 and 5.2 of \cite{nollasekiya} and Example 6.1 of \cite{ishiiitonolla}. The quotient variety $\mathbb{C}^2/D_{2n}$ can be realized as a subscheme of $D_{2n}\operatorname{-Hilb}(\mathbb{C}^3)$.
	
	We refer to the table of the normal bundles of the exceptional divisors (and their flops) $\mathcal{N}_{X/E}$ with their corresponding open covers in the same Theorems 5.1 and 5.2 of \cite{nollasekiya} to see which open sets cover the floppable $(-1,-1)$ curve.
	
	From the notation in \cite{nollasekiya}, under a suitable generic parameter $\theta^i$ satisfying the inequalities of Theorem 6.4 of \cite{nollasekiya}, we define $Y_i' := \mathcal{M}_{\theta^i}(\mathbb{C}^2)$ realized as moduli space of $G$-constellations of $\mathbb{C}^2$ which can be embedded in $X_{0...i} := \mathcal{M}_{\theta^i}(\mathbb{C}^3)$ realized as moduli space of $G$-constellations of $\mathbb{C}^3$.
	
	\begin{tikzcd}
		Y_{m-1}' = Y_{max} = \mathbb{Z}_2\operatorname{-Hilb}(X_1) \arrow[hookrightarrow]{r} \arrow[dashed, "T_{m-1}"]{d} & X_{0...(m-1)} = \mathbb{Z}_2\operatorname{-Hilb}(\mathbb{Z}_n\operatorname{-Hilb}(\mathbb{C}^3)) \arrow[dashed]{d}\\
		Y_{m-2}'  \arrow[hookrightarrow]{r} \arrow[dashed, "T_{m-2}"]{d} & X_{0...(m-2)} \arrow[dashed]{d}\\
		\vdots \arrow[dashed, "T_1"]{d} & \vdots \arrow[dashed]{d}\\
		Y_0' \arrow[hookrightarrow]{r} \arrow[dashed, "T_0"]{d} & X_{0...0} \arrow[dashed]{d} \\
		\mathbb{C}^2/G \cong G\operatorname{-Hilb}(\mathbb{C}^2) \arrow[hookrightarrow]{r} & G\operatorname{-Hilb}(\mathbb{C}^3)\\
	\end{tikzcd}
	
	The number of flops from the iterated Hilbert scheme $X_{0...(m-1)}$ to the $G$-Hilbert scheme $G\operatorname{-Hilb}(\mathbb{C}^3)$ is the same as the number of exceptional divisors of the morphism $f_2: Y_{max} \rightarrow \mathbb{C}^2/G$.
	
	After each flop of $(-1,-1)$ curve, the number of exceptional divisors over the surface must either: (a) increase by one, (b) decrease by one, or (c) stays the same.
	
	Furthermore, after $m$ flops from $X_{0...(m-1)}$, we eventually reach the Hilbert scheme $G\operatorname{-Hilb}(\mathbb{C}^3)$, whose image of $Y_{max}$ must be $G\operatorname{-Hilb}(\mathbb{C}^2) \cong \mathbb{C}^2/G$. Thus, all of the birational transformations over the surfaces must decrease the number of exceptional curves by one.
	
	We examine the flop restricted to the surface, most especially the open cover containing the exceptional divisor $E_i$, in Figure \ref{flopflop}.
	
	\begin{figure}
		\begin{center}
			\includegraphics[scale=0.40]{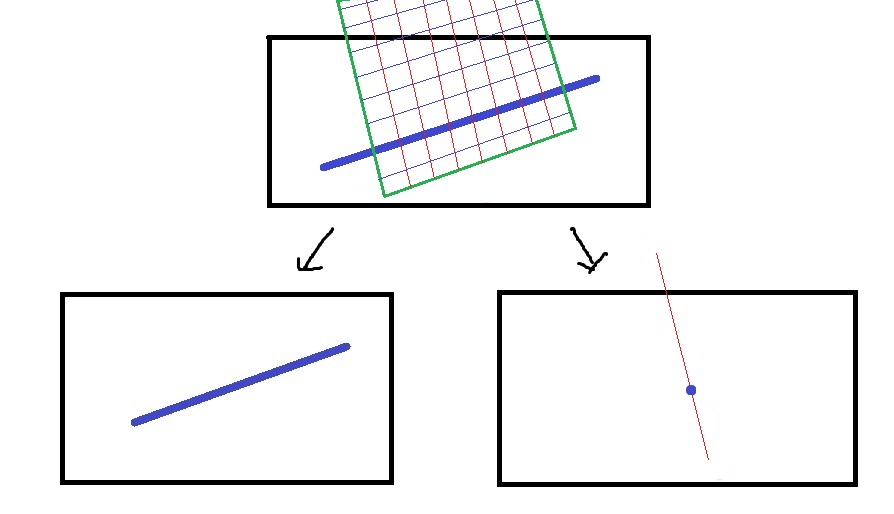}
			\caption{The process of flop in relation to the surface containing the exceptional curve}\label{flopflop}
		\end{center}
	\end{figure}
	
	Because the surface contains the $(-1,-1)$ flopping curve over the threefold, blowing up with the flopping curve as the center in the threefold, over the surface, the exceptional curve remains the same. Over the threefold, this produces the $\mathbb{P}^1 \times \mathbb{P}^1$ exceptional surface. Contracting in the other direction, this contracts the exceptional divisor on the surface.
	
	Because each of the surfaces $Y_i'$ are smooth, this now implies that each of the broken arrows over the two-dimensional variants are blowdown morphisms of a $(-1)$-curve, and comparing with Corollary \ref{resev}, every resolution dominated by $Y_{max}$ can be realized as a moduli space of $G$-constellations, which proves the main theorem.
\end{proof}

\begin{remark}\label{flopcoordinate}
	We fix the following notations for the open sets for odd $n$:
	\begin{align*}
		U_1' &\cong \operatorname{Spec}(\mathbb{C}[\frac{z}{f_2}, f_1, xy])\\
		U_i' &\cong \operatorname{Spec}\left(\mathbb{C}\left[\frac{(xy)^{i-1}z}{f_2}, \frac{f_1}{(xy)^{i-1}}, xy\right]\right) \mbox{ for } i \leq m + 1 \\ 
		U_i &\cong \operatorname{Spec}\left(\mathbb{C}\left[\frac{(xy)^{i-1}z}{f_2}, \frac{f_2}{(xy)^{i-2}z}, \frac{zf_1}{f_2}\right]\right), \mbox{ for } i \leq m + 1\\
		U_i'' &\cong \operatorname{Spec}\left(\mathbb{C}\left[\frac{zf_1}{f_2}, \frac{(xy)^i}{f_1}, \frac{f_1}{(xy)^{i-1}}\right]\right) \mbox{ for } i \leq m \\
		U_{m+2} &\cong \operatorname{Spec}\left(\mathbb{C}\left[z^2, \frac{f_1}{(xy)^m}, \frac{f_2}{(xy)^m z}\right]\right)\\
		X_{0...i} &= \bigcup_{k=1}^{i+1} U_k'' \cup U_{i+2}' \cup \bigcup_{k=i+3}^{m+2} U_k   
	\end{align*}
	We fix the notations also for the open sets for even $n$:
	\begin{align*}
		U_i &\cong \operatorname{Spec}\left(\mathbb{C}\left[\frac{(xy)^{i-1}z}{f_1f_2}, \frac{f_1f_2}{(xy)^{i-2}z}, \frac{zf_1}{f_2}\right]\right), \mbox{ for } i \leq m\\
		U_{m+1} &\cong \operatorname{Spec}\left(\mathbb{C}\left[\frac{zf_2}{f_1}, \frac{f_1f_2}{(xy)^{m-1}z}, \frac{zf_1}{f_2}\right]\right)\\
		U_i' &\cong \operatorname{Spec}\left(\mathbb{C}\left[\frac{(xy)^{i-1}z}{f_1f_2}, \frac{f_1^2}{(xy)^{i-1}}, xy\right]\right) \mbox{ for } i \leq m\\
		U_{m+1}' &\cong \operatorname{Spec}\left(\mathbb{C}\left[\frac{zf_2}{f_1}, \frac{f_1^2}{f_2^2}, \frac{f_2^2}{(xy)^{m-1}}\right]\right)\\
		U_i'' &\cong \operatorname{Spec}\left(\mathbb{C}\left[\frac{zf_1}{f_2}, \frac{(xy)^i}{f_1^2}, \frac{f_1^2}{(xy)^{i-1}}\right]\right) \mbox{ for } i \leq m-1
	\end{align*}
	\begin{align*}
		U_m'' &\cong \operatorname{Spec}\left(\mathbb{C}\left[\frac{zf_1}{f_2}, \frac{f_2^2}{f_1^2}, \frac{f_1^2}{(xy)^{m-1}}\right]\right)\\
		V_i' &\cong \operatorname{Spec}\left(\mathbb{C}\left[\frac{(xy)^{i-2}z^2}{f_2^2}, xy, \frac{f_1f_2}{(xy)^{i-2}z}\right]\right) \mbox{ for } i \leq m\\
		V_{m+1}' &\cong \operatorname{Spec}\left(\mathbb{C}\left[\frac{(xy)^{m-1}z^2}{f_2^2}, \frac{f_2^2}{(xy)^{m-1}}, \frac{f_1f_2}{(xy)^{m-1}z}\right]\right)\\
		V_{m+2}' &\cong \operatorname{Spec}\left(\mathbb{C}\left[z^2, \frac{f_2^2}{(xy)^{m-1}}, \frac{f_1}{zf_2}\right]\right)\\
		V_i'' &\cong \operatorname{Spec}\left(\mathbb{C}\left[\frac{(xy)^{i-2}z^2}{f_2^2}, \frac{f_2^2}{(xy)^{i-3}z^2}, \frac{zf_1}{f_2}\right]\right) \mbox{ for } i \leq m + 1\\  
		V_{m+2}'' &\cong \operatorname{Spec}\left(\mathbb{C}\left[z^2, \frac{f_2^2}{(xy)^{m-1}z^2}, \frac{zf_1}{f_2}\right]\right)\\
		V_{m+3}'' &\cong \operatorname{Spec}\left(\mathbb{C}\left[z^2, \frac{f_1^2}{(xy)^{m-1}}, \frac{f_2}{zf_1}\right]\right)\\
		X_{0...i}^{m...(m-j)} &= \bigcup_{k=1}^{i+1} U_k'' \cup U_{i+2}' \cup \bigcup_{k=i+3}^{m-j} U_k \cup V_{m-j+1}' \cup \bigcup_{k=m-j+2}^{m+3} V_k''
	\end{align*}
	
	An explicit way to approach the theorem is to consider also the open sets so that $U_m'' \cup U_{m+1}'$ cover the exceptional divisor with coordinates \[ E_m : \begin{cases} 
		(f_1 : (xy)^m) & \mbox{for odd $n$} \\
		(f_1^2 : (xy)^m) & \mbox{for even $n$}
	\end{cases}
	\]
	
	From the proof of Theorems 5.1 and 5.2 of \cite{nollasekiya}, under the flopping transformation of the exceptional divisor $E_m$, the open sets $U_m'' \cup U_{m+1}'$ maps to the open sets $U_m' \cup U_{m+1}$ covering the flop of $E_m$.
	The coordinate of the flop of $E_m$ is \[ \hat{E}_m : \begin{cases} 
		((xy)^{m-1} z : f_2) & \mbox{for odd $n$} \\
		((xy)^{m-1} z : f_1f_2) & \mbox{for even $n$}
	\end{cases}
	\]
	so that on the new variety $X_{0...(m-2)}$, the exceptional divisor vanishes (or is contracted in the two dimensional case).
	
	We also note the gluing between the open sets $U_m''$ and $U_{m+1}'$ via: \begin{align*} 
		(\frac{zf_1}{f_2}, (xy)^m/f_1, f_1/(xy)^{m-1}) &\mapsto (\frac{zf_1}{f_2} \cdot \frac{(xy)^m}{f_1}, (\frac{(xy)^m}{f_1})^{-1}, \frac{(xy)^m}{f_1} \cdot \frac{f_1}{(xy)^{m-1}}) \\
		& \mbox{(for odd $n$)} \\
		(\frac{zf_1}{f_2}, \frac{f_1^2}{(xy)^{m-1}}, \frac{f_2^2}{f_1^2}) &\mapsto (\frac{zf_1}{f_2} \cdot \frac{f_2^2}{f_1^2}, (\frac{f_2^2}{f_1^2})^{-1}, \frac{f_1^2}{(xy)^{m-1}} \cdot \frac{f_2^2}{f_1^2})\\
		& \mbox{(for even $n$)}
	\end{align*}
	
	Thus, with the realization of the iterated Hilbert scheme over the surface as the locus from earlier, the loci $zf_1/f_2 = 0$ and $(xy)^mz/f_2 = 0$ collapses the open set $U_{m+1}$. The open cover implies that this is identical to the blow-down of the exceptional curve $\hat{E}_m$. In symbols, we have the commutative diagram:
	
	\begin{tikzcd}
		X_{0...(m-2)} \supset U_m' \cup U_{m+1} \supset \hat{E}_m & \arrow[dashed,l, "\Psi_{m-1}"] E_m \subset U_m'' \cup U_{m+1} \subset X_{0...(m-1)}\\
		Y_{m-2}' \supset T_{m-1}(E_m) = pt \arrow[u, "i"] & \arrow[l, "T_{m-1}"] E_m \subset Y_{max}\arrow[u, "i"]
	\end{tikzcd}
	
	From the flop of $E_m$, we obtain another crepant resolution. This time, we do this similarly for the open sets $U_{i}'' \cup U_{i+1}'$ for $1 \leq i \leq m-1$, which gives another crepant resolution. This works for both cases.
\end{remark}

\section{The Tautological Bundles and the McKay Correspondence}
In the next sections, we explicitly construct a McKay correspondence by investigating the exceptional divisors of the maximal resolution $Y_{max} \rightarrow \mathbb{C}^2/D_{2n}$ using the realization of the maximal resolution as a moduli space of $D_{2n}$- constellations.

In this section, in the spirit of \cite{gonzverd}, we consider the tautological bundles and consider the stacky descriptions (or equivalently their $\mathbb{Z}_2$-equivariant sheaves) to construct the McKay correspondence for the dihedral groups. The main result in this section is the description parallel to the proof of Theorem 1.11 in \cite{artverd}, in particular the description of a rank $n$ indecomposable reflexive module as an extension of two vector bundles (one of which is a line bundle) of the form $0 \rightarrow O_{\tilde{X}}^{\oplus (n-1)} \rightarrow \tilde{M} \rightarrow O_{\tilde{X}}(\tilde{D}) \rightarrow 0$, whose Chern class $c_1(\tilde{M})$ of the locally free sheaf $\tilde{M}$ corresponds to a vertex of the Dynkin diagram associated to $\tilde{X}$.

We also insert a comparison between the tautological sheaves on the stack and the tautological sheaves on the coarse moduli space to show how they behave differently in those spaces. Together with the results on the final section, we show evidences why constructing a McKay correspondence is more plausible over the stack than on the coarse moduli space.

We consider the representations of $\mathbb{Z}_n$ and $D_{2n}$ as described by the following character tables \ref{tab:my_label1} and \ref{tab:my_label2}. There are four one-dimensional representations $\rho_0, \rho_0', \rho_{n/2}, \mbox{and } \rho_{n/2}'$. For $1 \leq j \leq n/2$ (for even $n$) or $1 \leq j \leq (n-1)/2$ (for odd $n$), the two-dimensional representations $\rho_j$ are defined by their corresponding matrix representations: $\rho_i = \left\langle \tau :=  \begin{bmatrix}
	0 & 1 \\
	1 & 0
\end{bmatrix}, \sigma^i := \begin{bmatrix}
	\epsilon^i & 0 \\
	0 & \epsilon^{-i}
\end{bmatrix} (\epsilon^n = 1)  \right\rangle$.

\begin{table}[H]
	\centering
	\caption{The Irreducible Representations of $\mathbb{Z}_n$}
	\label{tab:my_label1}
	\begin{tabular}{c|c|c}
		\backslashbox{\mbox{Representation}}{\mbox{Conjugacy Class}} & $1$ & $\sigma^i$ \\\hline
		$\epsilon_0$ & $1$ & $1$\\
		$\epsilon_j$ & $1$ & $\epsilon^{ij}$
	\end{tabular}
	
	where $1 \leq j \leq n-1$ and $0 \leq i \leq n-1$ are integers.
\end{table}
\begin{table}[H]
	\centering
	\caption{The Irreducible Representations of $D_{2n}$}
	\label{tab:my_label2}
	\begin{tabular}{c|c|c|c}
		\backslashbox{\mbox{Representation}}{\mbox{Conjugacy Class}} & $1$ & $\tau$ & $\sigma^i$ \\\hline
		$\rho_0$ & $1$ & $1$ & $1$\\
		$\rho_0'$ & $1$ & $-1$ & $1$\\
		$\rho_j$ & $2$ & $0$ & $\epsilon^{ij} + \epsilon^{-ij}$\\
		$\rho_{n/2}$ \mbox{($n$ even)} & $1$ & $1$ & $(-1)^i$ \\
		$\rho_{n/2}'$ \mbox{($n$ even)} & $1$ & $-1$ & $(-1)^i$
	\end{tabular}
	
	where $1 \leq j \leq \frac{n-1}{2}$ and $0 \leq i \leq n-1$ are integers.
\end{table}

Consider the diagram:
\[
\begin{tikzcd}
	& \mathcal{Z} \subset X_1 \times \mathbb{C}^2 \arrow[dl, "p_{X_1}"] \arrow[dr,"p_{\mathbb{C}^2}"] \\
	X_1 && \mathbb{C}^2
\end{tikzcd}
\]

where $\mathcal{Z} \subset X_1 \times \mathbb{C}^2$ is the universal subscheme and $p_{X_1}$ and $p_{\mathbb{C}^2}$ are natural projections.

Since the group $D_{2n}$ acts on each of the schemes $X_1$ and $\mathbb{C}^2$, by defining $[X/G]$ as the (quotient) stack associated to the scheme $X/G$, we can construct the diagram:

\[
\begin{tikzcd}
	& \left[\mathcal{Z}/D_{2n}\right] \subset \left[(X_1 \times \mathbb{C}^2) / D_{2n} \right] \arrow[dl, "p_{\left[X_1/ D_{2n}\right]}"] \arrow[dr,"p_{\left[\mathbb{C}^2/ D_{2n}\right]}"] \\
	\left[X_1/ D_{2n}\right] && \left[\mathbb{C}^2 / D_{2n} \right]
\end{tikzcd}
\]

Because $\mathbb{Z}_n$ acts trivially on $X_1$, there is a natural morphism $\phi: \left[X_1/D_{2n} \right] \rightarrow \left[X_1/\mathbb{Z}_2 \right]$, so that the pushforward morphism $\phi_*$ sends a $D_{2n}$-equivariant coherent sheaf $\mathcal{F}$ on $X_1$ to the $\mathbb{Z}_2$-equivariant coherent sheaf $\mathcal{F}^{\mathbb{Z}_n}$.

We also use the fact that the category of (quasi-)coherent sheaves on the quotient stack $[X/G]$ is equivalent to the category of $G$-equivariant (quasi-)coherent sheaves on the scheme $X$, i.e. $(Q)Coh([X/G]) \cong (Q)Coh^G(X)$.

By Theorem \ref{maxemb}, the maximal resolution and the quotient variety $Y_1$ are identical. This means that the corresponding boundary divisors are identical as well. This means that by considering the 2nd root stack $$\sqrt{(O_{\mathbb{Z}_n\operatorname{-Hilb}(\mathbb{C}^2)/\mathbb{Z}_2}(\ceil*{B'}), 1_{\ceil*{B'}})/(\mathbb{Z}_n\operatorname{-Hilb}(\mathbb{C}^2)/\mathbb{Z}_2)},$$ we obtain the isomorphism between stacks $\mathcal{Y} \cong [\mathbb{Z}_n\operatorname{-Hilb}(\mathbb{C}^2)/\mathbb{Z}_2]$.  

Defining $\mathcal{Z} \subset X_1 \times \mathbb{C}^2$ as the universal family of $\mathbb{Z}_n$-constellations (or clusters) and consider the diagram whose $p_{[X_1/D_{2n}]}$ and $p_{[\mathbb{C}^2/D_{2n}]}$ are natural projections: \[
\begin{tikzcd}
	& \left[\mathcal{Z}/D_{2n}\right] \subset \left[(X_1 \times \mathbb{C}^2) / D_{2n} \right] \arrow[dl, "p_{\left[X_1/ D_{2n}\right]}"] \arrow[dr,"p_{\left[\mathbb{C}^2/ D_{2n}\right]}"] \\
	\left[X_1/ D_{2n}\right] && \left[\mathbb{C}^2 / D_{2n} \right]
\end{tikzcd}
\]

We have the Fourier-Mukai transforms:
\begin{align*}
	\Phi: D([\mathbb{C}^2/D_{2n}]) &\rightarrow D(\mathcal{Y}) \\
	\mathfrak{G} &\mapsto \phi_*(Rp_{[\mathbb{Z}_n\operatorname{-Hilb}(\mathbb{C}^2)/D_{2n}]*} (p_{[\mathbb{C}^2/D_{2n}]}^*(\mathfrak{G}) \otimes O_{[\mathcal{Z}/D_{2n}]}))\\
	\Psi: D(\mathcal{Y}) &\rightarrow D([\mathbb{C}^2/D_{2n}]) \\
	\epsilon &\mapsto R(p_{[\mathbb{C}^2/D_{2n}]*}) (p_{[\mathbb{Z}_n\operatorname{-Hilb}(\mathbb{C}^2)/D_{2n}]}^*(\phi^*(\epsilon)) \otimes det(\rho_{nat}) \\
	& \otimes O_{[\mathcal{Z}/D_{2n}]}^{\lor}[2])
\end{align*}

The functors $\Psi$ and $\Phi$ are equivalences via Theorem 4.1 of \cite{ishiiueda}.

We define the tautological sheaf associated to a representation $\rho$ of $D_{2n}$ as $$\hat{\mathcal{R}}_{\rho} := \Phi(O_{\mathbb{C}^2} \otimes \rho^{\lor}) = (p_{\mathbb{Z}_n\operatorname{-Hilb}(\mathbb{C}^2)*}(O_{\mathcal{Z}}) \otimes \rho^{\lor})^{\mathbb{Z}_n} = \left( [\oplus_{\epsilon \in \operatorname{Irr}(\mathbb{Z}_n)} \mathcal{R}_{\epsilon} \otimes \epsilon] \otimes \rho^{\lor} \right)^{\mathbb{Z}_n},$$
which are $\mathbb{Z}_2$-equivariant locally free sheaves since $\mathcal{R}_{\epsilon}$ are locally free sheaves, and the $\mathbb{Z}_2$-sheaf structure came from the induced representation from $\mathbb{Z}_n$ to $D_{2n}$.

We consider also the diagram:
\[
\begin{tikzcd}
	& \mathcal{Y} \arrow[dr,"\pi"] \\
	\mathbb{Z}_n\operatorname{-Hilb}(\mathbb{C}^2) \arrow[ur,"f"] \arrow[rr,"p"] && \mathbb{Z}_n\operatorname{-Hilb}(\mathbb{C}^2)/\mathbb{Z}_2
\end{tikzcd}
\]

to define the corresponding tautological sheaves on $\mathcal{M}_{\theta} = \mathbb{Z}_n\operatorname{-Hilb}(\mathbb{C}^2)/\mathbb{Z}_2$, for some $\theta \in \Theta$.

The tautological bundle $\mathcal{R} := p_{X_1*}(O_{\mathcal{Z}})$ on $\mathbb{Z}_n\operatorname{-Hilb}(\mathbb{C}^2)$ decomposes as:

$$\mathcal{R} = \bigoplus_{\epsilon \in \operatorname{Irr}(\mathbb{Z}_n)} \mathcal{R}_{\epsilon}^{\circ} \otimes \epsilon.$$

By considering $\mathcal{R}_{\epsilon}^{\circ} := (\mathcal{R} \otimes \epsilon^{\lor})^{\mathbb{Z}_n}$ and $\mathcal{R}_{\rho}^{\circ} := (p_*(\mathcal{R}) \otimes \rho^{\lor})^{D_{2n}}$ as subsheaves of $K(\mathbb{C}^2) \otimes \epsilon_i^{\lor}$ and $K(\mathbb{C}^2) \otimes \rho_i^{\lor}$, respectively, we have the following images of the tautological bundles on $\mathbb{Z}_n\operatorname{-Hilb}(\mathbb{C}^2)$ under the pushforward $p_*$:
\begin{align*}
	p_*(\mathcal{R}_{\epsilon_i} \otimes \epsilon_i \oplus \mathcal{R}_{\epsilon_{n-i}} \otimes \epsilon_{n-i}) &= (\mathcal{R}_{\rho_i} \otimes \rho_i)^{\oplus 2} \mbox{ } \mbox{ (for $i \neq n/2$)}\\
	p_*(\mathcal{R}_{\epsilon_0} \otimes \epsilon_0) &= \mathcal{R}_{\rho_0} \otimes \rho_0 \oplus \mathcal{R}_{\rho_0'} \otimes \rho_0'\\
	p_*(\mathcal{R}_{\epsilon_{n/2}} \otimes \epsilon_{n/2}) &= \mathcal{R}_{\rho_{n/2}} \otimes \rho_{n/2} \oplus \mathcal{R}_{\rho_{n/2}'} \otimes \rho_{n/2}'
\end{align*}

We collect some lemmas in order to establish a parallel statement of the correspondence:

\begin{lemma}\label{finp}
	For a (Cartier) divisor $D$ on $X_1$, the followings hold:
	\begin{enumerate}
		\item $\operatorname{det}(p_*(O_{X_1}(D)) = \operatorname{det}(p_*(O_{X_1})) \otimes O_{Y_1}(p_*(D))$.
		\item $\operatorname{det}(p_*(O_{X_1})) = O_{Y_1}(L)$, where $L$ is a $\mathbb{Q}$-divisor satisfying $2L = -(B_1 + B_2)$ (for even $n$) and $2L = -B_3$ (for odd $n$). Please refer again to (\ref{eqb1}) and (\ref{eqb2}) for the definition of the ramification divisors. This shows that the sheaf $O_{Y_1}(L)$ is a line bundle.
		\item $p_*(O_{X_1}) = O_{Y_1} \oplus O_{Y_1}(L)$.
	\end{enumerate}
\end{lemma}

\begin{proof}
	The statements (1) and (2) are essentially (Ch. IV.2, Ex. 2.6) of \cite{hartshorne1}, but we provide the complete proof. In the following, we distinguish $K_X$ as the canonical divisor, and $\omega_X = O_X(K_X)$ is the canonical sheaf.
	
	For an effective divisor $D_1$, consider the exact sequence: $$0 \rightarrow O_{X_1} \rightarrow O_{X_1}(D_1) \rightarrow O_{D_1}(D_1) \rightarrow 0.$$ Applying the push-forward $p_*$ gives the exact sequence:
	$$0 \rightarrow p_*(O_{X_1}) \rightarrow p_*(O_{X_1}(D_1)) \rightarrow p_*(O_{D_1}(D_1)) \rightarrow 0.$$
	Taking the Chern classes over the exact sequence gives statement (1) for the effective divisor $D_1$.
	
	If $D$ is not effective, we can write $D$ as $D = D_1 - D_2$, where both $D_1$ and $D_2$ are effective. We have shown the statement for the effective divisor $D_1$. For the divisor $-D_2$, we apply the similar sequence:
	$$0 \rightarrow O_{X_1}(-D_2) \rightarrow O_{X_1} \rightarrow O_{D_2} \rightarrow 0.$$
	
	Applying the push-forward $p_*$ and taking the Chern classes gives the statement (1) for any divisor $D = D_1 - D_2$.
	
	By the duality for a finite flat morphism (Ch. III.6, Ex. 6.19 of \cite{hartshorne1}), with $\omega_{Y_1}$ being the dualizing sheaf of $Y_1$, so that  $\omega_{X_1} \cong p^{!}(\omega_{Y_1}) := \reallywidetilde{\mathcal{H}om(p_*(O_{X_1}), \omega_{Y_1})}$ (where for a $p_*(O_{X_1})$-module $\mathcal{M}$, the sheaf $\reallywidetilde{\mathcal{M}}$ is the associated $O_{X_1}$-module, as seen in Ch. II.5, Ex. 5.17 of \cite{hartshorne1}), is also a dualizing sheaf for $X_1$, and we have the isomorphism:
	$$p_*(O_{X_1}) \cong p_*\mathcal{H}om_{X_1}(\omega_{X_1}, \omega_{X_1}) \cong \mathcal{H}om_{Y_1}(p_*(\omega_{X_1}), \omega_{Y_1}) \cong (p_*(\omega_{X_1}))^{\lor} \otimes \omega_{Y_1}.$$
	
	Taking the Chern classes gives the isomorphism:
	$$\operatorname{det}(p_*(\omega_{X_1})) \cong \operatorname{det}(p_*(O_{X_1}))^{-1} \otimes \omega_{Y_1}^{\otimes 2}.$$
	
	By plugging $D = K_{X_1}$ on statement (1), we compare different expressions for $\operatorname{det}(p_*(\omega_{X_1}))$:
	\begin{align*}
		\operatorname{det}(p_*(O_{X_1}))^{-1} \otimes \omega_{Y_1}^{\otimes 2} &\cong \operatorname{det}(p_*(O_{X_1})) \otimes O_{Y_1}(p_*(K_{X_1}))\\
		\operatorname{det}(p_*(O_{X_1}))^{\otimes 2} &\cong O_{Y_1}(2K_{Y_1} - p_*(K_{X_1})).
	\end{align*}
	
	From the relation $K_{X_1} = p^*(K_{Y_1} - L)$, where $-2L = B_1 + B_2$ (in the even $n$ case) and $-2L = B_3$ (in the odd $n$ case), taking the pushforward $p_*$, we get statement (2) thereafter:
	\begin{align*}
		p_*(K_{X_1}) &= 2K_{Y_1} - 2L\\
		\operatorname{det}(p_*(O_{X_1}))^{\otimes 2} &\cong O_{Y_1}(2L).
	\end{align*}
	Because $\mathbb{Z}_2$ acts trivially on $Y_1$ and $p$ is a finite flat morphism, $p_*(O_{X_1})$ decomposes into a $\mathbb{Z}_2$-invariant locally free sheaf and $\mathbb{Z}_2$-anti-invariant locally free sheaf; or more precisely:
	$$p_*(O_{X_1}) = \mathcal{F}_0  \oplus \mathcal{F}_1 \otimes \delta,$$
	where $\delta$ is the nontrivial representation of $\mathbb{Z}_2$.
	Certainly, $\mathcal{F}_0 = O_{Y_1}$, and it follows that $\mathcal{F}_1 = \operatorname{det}(p_*(O_{X_1})) = O_{Y_1}(L)$ showing (3).
\end{proof}

\begin{lemma}\label{refdivisor}
	There exists a (Weil) divisor $D$ on the maximal resolution $Y_{max}$ such that:
	\begin{enumerate}
		\item $D$ is transversal to the exceptional divisor $E_k$; $D \cdot E_k = 1$, for some $1 \leq k \leq m$.
		\item $D \cdot E_j = 0, j \neq k$.
		\item For even $n$, $D$ does not coincide with either of $B_1$ or $B_2$; $D \neq B_1, B_2$.
	\end{enumerate}
\end{lemma}

\begin{proof}
	We prove this in odd $n$ case because the same argument will work in the even $n$ case.
	
	We refer to the open sets of the maximal resolution from Remark \ref{flopcoordinate}.
	
	Consider the locus in $\mathbb{C}^2/G$ defined by the equation $W_k: f_1 - (xy)^k = 0; (1 \leq k \leq m = \frac{n-1}{2})$.
	
	We shall illustrate for $k = 1$ as the rest of the cases can be performed very similarly.
	
	On $U_1''$, the locus $W_1$ is the line $xy/f_1 = 1$, which only intersects the coordinate axis at $(xy/f_1, f_1) = (1,0)$, which corresponds to a point on the exceptional divisor $E_1$.
	
	By the gluing between the open sets $U_1''$ and $U_2''$, this identifies the same point $(xy/f_1, f_1) = (1,0) = (f_1/xy, (xy)^2/f_1)$.
	
	Once again, by the gluing between the open sets $U_2''$ and $U_3''$ (and so on), the following equation defines the locus, which has no intersection with the coordinate axes of $U_i'', i \neq 1,2;$ which implies that there is no intersection with the exceptional divisors other than $E_1$:
	$$0 = \frac{f_1}{xy} - 1 = \left(\frac{(xy)^{k+1}}{(f_1)}\right)^{k-1} \cdot \left(\frac{(f_1)}{(xy)^k}\right)^{k-1} \cdot \frac{f_1}{(xy)^k} - 1.$$
\end{proof}

\begin{remark}
	For odd $n$, the boundary divisor $B_3$ is defined by a quadratic equation via Remark \ref{oddboundary}, hence, any divisor transversal to $E_m$ must intersect the boundary divisor.
\end{remark}

\begin{lemma}\label{tautdes}
	From the equivalence of categories $Coh^{\mathbb{Z}_2}(X_1) \cong Coh(\mathcal{Y})$:
	\[ O_{X_1} \otimes \delta \mapsto O_{\mathcal{Y}}(\mathcal{C}) := \begin{cases} O_{\mathcal{Y}}(\mathcal{B}_3 - \mathcal{D}) & \text{if } $n$ \text{ is odd}\\
		O_{\mathcal{Y}}(\mathcal{B}_1 - \mathcal{B}_2) & \text{if } $n$ \text{ is even}
	\end{cases}\]
	where $\mathcal{B}_i$ is a prime divisor on $\mathcal{Y}$ such that $2\mathcal{B}_i = \pi^{-1}(B_i)$ (this is the stacky locus); and $\mathcal{D} := \pi^{-1}(D)$, where $D$ satisfies the conditions in Lemma \ref{refdivisor}. (Refer to (\ref{eqb1}) and (\ref{eqb2}) for the definition of $B$.)
\end{lemma}

\begin{proof}
	We comment first on the derivation for this description. The $\mathbb{Z}_2$-equivariant sheaf $O_{X_1} \otimes \delta$ is a torsion element of $Pic^{\mathbb{Z}_2}(X_1)$ of order $2$, so that the corresponding sheaf on the stack is also a torsion element of $Pic(\mathcal{Y})$ of order $2$. Because of the following relations:
	\begin{align*}
		O_{\mathcal{Y}} (2\mathcal{B}_3 - 2\mathcal{D}) &= \pi^*(O_{Y_1}(B_3 - 2D)) \cong \pi^*(O_{Y_1}) = O_{\mathcal{Y}} \mbox{ for odd $n$ }\\
		O_{\mathcal{Y}} (2\mathcal{B}_1 - 2\mathcal{B}_2) &= \pi^*(O_{Y_1}(B_1 - B_2)) \cong \pi^*(O_{Y_1}) = O_{\mathcal{Y}} \mbox{ for even $n$ }
	\end{align*}
	so that $O_{\mathcal{Y}}(\mathcal{B}_3 - \mathcal{D})$ and $O_{\mathcal{Y}}(\mathcal{B}_1 - \mathcal{B}_2)$ are torsion elements of $Pic(\mathcal{Y})$. It is important to point out that these are distinct divisors by considering $\mathcal{B}_i = [B_i/\mathbb{Z}_2]$. Another way is to consider the Fourier-Mukai images of $\mathcal{B}$ under $\Phi$ via Theorem \ref{stackfm}. The stabilizer groups of $\mathcal{B}_i$ and $\mathcal{D}$ are $\mathbb{Z}_2$ and $\{e\}$ (except for the intersection point with $\mathcal{B}_i$), respectively. 
	
	We further note that such $D$ exists by the previous Lemma \ref{refdivisor}.
	
	We consider the commutative diagram. This particular diagram defines the equivalence between the two categories. Our main argument here is to trace via the diagram.
	
	\begin{tikzcd}
		X_1 \times \mathbb{Z}_2 \arrow[r, "pr_{X_1}"] \arrow[d, "a"]& X_1 \arrow[d, "\pi"]\\
		X_1 \arrow[r, "\pi"]& \mathcal{Y}
	\end{tikzcd}
	
	which gives the isomorphism $\alpha: a^*\pi^*(O_{\mathcal{Y}}(\mathcal{C})) \xrightarrow{\sim} pr_{X_1}^*\pi^*(O_{\mathcal{Y}}(\mathcal{C}))$.
	
	By considering the fiber at each point, the stabilizer group of $\mathcal{C}$ is $\mathbb{Z}_2$ and for the other points being $\{e\}$. This induces the non-trivial action of $\mathbb{Z}_2$, and thus $\alpha$ is not identity.
\end{proof}

\begin{theorem}\label{mainthm2}
	The tautological bundles on the stack $\mathcal{Y} = [\mathbb{Z}_n\operatorname{-Hilb}(\mathbb{C}^2)/\mathbb{Z}_2]$ are described by the following: \begin{center}
		\captionof{table}{The Tautological Bundles for odd $n$}
		\begin{tabular}{|c c c|} 
			\hline
			Tautological Sheaf & Description & Chern Class \\ [0.5ex] 
			\hline
			$\hat{\mathcal{R}}_{\rho_0}$ & $\mathcal{O}_{\mathcal{Y}}$ & $0$ \\ 
			\hline
			$\hat{\mathcal{R}}_{\rho_0'}$ & $\mathcal{O}_{\mathcal{Y}}(\mathcal{B}_3 - \mathcal{D})$ & $\mathcal{B}_3 - \mathcal{D}$ \\
			\hline
			$\hat{\mathcal{R}}_{\rho_i}$ (rank 2) & $0 \rightarrow \mathcal{O}_{\mathcal{Y}} \xrightarrow{i} \hat{\mathcal{R}}_{\rho_i} \xrightarrow{pr} \mathcal{O}_{\mathcal{Y}}(\mathcal{D}_i + \mathcal{B}_3 - \mathcal{D}) \rightarrow 0$ & $\mathcal{D}_i + \mathcal{B}_3 - \mathcal{D}$ \\
			\hline
		\end{tabular}
		
		\captionof{table}{The Tautological Bundles for even $n$}
		\begin{tabular}{|c c c|} 
			\hline
			Tautological Sheaf & Description & Chern Class \\ [0.5ex] 
			\hline
			$\hat{\mathcal{R}}_{\rho_0}$ & $\mathcal{O}_{\mathcal{Y}}$ & $0$ \\ 
			\hline
			$\hat{\mathcal{R}}_{\rho_0'}$ & $\mathcal{O}_{\mathcal{Y}}(\mathcal{B}_1 -\mathcal{B}_2)$ & $\mathcal{B}_1 -\mathcal{B}_2$ \\
			\hline
			$\hat{\mathcal{R}}_{\rho_i}$ (rank 2) & $0 \rightarrow \mathcal{O}_{\mathcal{Y}} \xrightarrow{i} \hat{\mathcal{R}}_{\rho_i} \xrightarrow{pr} \mathcal{O}_{\mathcal{Y}}(\mathcal{D}_i + \mathcal{B}_1 - \mathcal{B}_2) \rightarrow 0$ & $\mathcal{D}_i + \mathcal{B}_1 - \mathcal{B}_2$ \\
			\hline
			$\hat{\mathcal{R}}_{\rho_{(n/2)}}$ & $\mathcal{O}_{\mathcal{Y}}(\mathcal{B}_1)$ & $\mathcal{B}_1$ \\
			\hline
			$\hat{\mathcal{R}}_{\rho_{(n/2)}'}$ & $\mathcal{O}_{\mathcal{Y}}(\mathcal{B}_2)$ & $\mathcal{B}_2$ \\ [1ex] 
			\hline
		\end{tabular}
	\end{center}
	
	The rank one tautological bundles on the stack are uniquely determined by their Chern classes; and the rank two tautological bundles are determined by an extension of two line bundles. Furthermore, there is only one possible non-trivial extension class, making these descriptions unique.
\end{theorem}

\begin{proof}
	We write an exact sequence involving (rank two) tautological sheaves $\tilde{\mathcal{R}}_{\rho_i}$ (as $\mathbb{Z}_2$-equivariant sheaves) similar to the proof of Theorem 1.11 of \cite{artverd}:
	$$0 \rightarrow \mathcal{O}_{X_1} \xrightarrow{i} \mathcal{O}_{X_1}(\tilde{D}_i) \oplus \mathcal{O}_{X_1}(g \cdot \tilde{D}_{i}) = \tilde{\mathcal{R}}_{\rho_i} \xrightarrow{pr} \mathcal{O}_{X_1}(\tilde{D}_i + g \cdot \tilde{D}_{i}) \otimes \delta \rightarrow 0,$$
	where $i$ is the inclusion, $\delta$ is the nontrivial representation of $\mathbb{Z}_2$ (and $g$ correspond to the non-trivial element of $\mathbb{Z}_2$), and $pr(h_1, h_2) = h_1-h_2$.
	
	We show that this extension over $X_1$ is unique, which implies on the tautological sheaves on the stack. We let $F = \Sigma \tilde{E}_i$ be the fundamental cycle. In the following, $\mathbb{Z}_2\operatorname{-Ext}^1_X(-,-)$ is defined as the $\mathbb{Z}_2$-invariant part of $Ext_X^1(-,-)$.
	\begin{align*}
		\mathbb{Z}_2\operatorname{-Ext}_{\mathcal{O}_{X_1}}^1(\mathcal{O}_{X_1}(\tilde{D}_i + g \cdot \tilde{D}_{i}) \otimes \delta, \mathcal{O}_{X_1}) &=\mathbb{Z}_2\operatorname{-Ext}_{\mathcal{O}_{X_1}}^1(\mathcal{O}_{X_1}, \mathcal{O}_{X_1}(- \tilde{D}_i - g \cdot \tilde{D}_{i}) \\ 
		&\otimes \delta)\\
		&= H^1(X_1, \mathcal{O}_{X_1}(- \tilde{D}_i - g \cdot \tilde{D}_{i}) \otimes \delta)^{\mathbb{Z}_2}\\
		&= H^1(F, \mathcal{O}_{X_1}(- \tilde{D}_i - g \cdot \tilde{D}_{i})|_F)^{\mathbb{Z}_2}\\
		&= H^1(F, \mathcal{O}_F(- \tilde{D}_i - g \cdot \tilde{D}_{i}))^{\mathbb{Z}_2}\\
		&\cong \mathbb{C}
	\end{align*}
	which is one dimensional over $\mathbb{C}$. Thus, there is only one possible non-trivial extension (up to scalar).
	
	From the Lemma \ref{tautdes}, this corresponds over the global quotient stack $\mathcal{Y}$:
	
	$$0 \rightarrow \mathcal{O}_{\mathcal{Y}} \xrightarrow{i} \hat{\mathcal{R}}_{\rho_i} \xrightarrow{pr} \mathcal{O}_{\mathcal{Y}}(\mathcal{D}_i + \mathcal{C}) \rightarrow 0,$$
	where $\hat{\mathcal{R}}_{\rho_i}$ and $\mathcal{O}_{\mathcal{Y}}(\mathcal{D}_i)$ is the corresponding image in the isomorphism of the tautological sheaf $\tilde{\mathcal{R}}_{\rho_i}$ and the invertible sheaf $\mathcal{O}_{X_1}(\tilde{D_i} + g \cdot \tilde{D_i})$, respectively.
	
	The rank one tautological sheaves $\hat{\mathcal{R}}_{\rho_0'}, \hat{\mathcal{R}}_{\rho_{n/2}}, \hat{\mathcal{R}}_{\rho_{n/2}'}$ are obtained via the commutative diagram, where $\Phi^{\mathbb{Z}_n}$ is the equivalence defined in \cite{bridgelandkingreid}: 
	
	\begin{tikzcd}
		D(coh^{D_{2n}}(\mathbb{C}^2))  \arrow[d, "\operatorname{for}"] \arrow[r, "\Phi"] & D(coh^{\mathbb{Z}_2}(X_1)) \arrow[d, "\operatorname{for}"]\\
		D(coh^{\mathbb{Z}_n}(\mathbb{C}^2)) \arrow[r, "\Phi^{\mathbb{Z}_n}"] & D(coh(X_1))
	\end{tikzcd}
	
	We demonstrate this for $\rho_{n/2}$ and the rest of the cases are similar. From the definition of $\hat{\mathcal{R}}_{\rho_{n/2}}$, this must come from a $\mathbb{Z}_2$-lift of $(\Phi_{\mathbb{Z}_n} \circ \operatorname{for})(O_{\mathbb{C}^2} \otimes \rho_{n/2}) = \tilde{\mathcal{R}}_{\epsilon_{n/2}}$, which must be a $\mathbb{Z}_2$-line bundle with degree $1$. But the only divisors that satisfy such property are $\tilde{B}_1$ and $\tilde{B}_2$. Since $\tilde{B}_1$ is invariant under the action of $\rho_{n/2}$, this describes the tautological sheaf $\tilde{\mathcal{R}}_{\rho_{n/2}} = O_{X_1}(\tilde{B}_1)$. In addition, $\tilde{\mathcal{R}}_{\rho_{n/2}'} = \tilde{\mathcal{R}}_{\rho_{n/2}} \otimes \delta$, which gives the complete description for the tautological sheaves on $X_1$.
	
\end{proof}

\begin{remark}
	The theory of Chern classes also works for Deligne-Mumford stacks via the theory of Chow groups with rational coefficients. This is shown in Section 3 of \cite{vistoli}.
\end{remark}

We now compare the tautological sheaves over the stack, and the bundles over the coarse moduli space. We will see that the rank two bundles split over the coarse moduli space.

Taking the pushforward and the $\mathbb{Z}_2$-invariants of the exact sequence above, we get the exact sequence on $Y_1$:
$$0 \rightarrow O_{Y_1} \rightarrow \mathcal{R}_{\rho_i} \rightarrow O_{Y_1}(D_i + L) \rightarrow 0,$$
where $D_i$ is the image $p(\tilde{D_i})$.

\begin{lemma}\label{split}
	The exact sequence over $Y_1$ is split. Equivalently: $$H^1(Y_1, \mathcal{O}_{Y_1}(L + D_i)) = 0.$$
\end{lemma}

\begin{proof}
	Let $Z$ be the fundamental cycle of the maximal resolution $Y_{max} \rightarrow \mathbb{C}^2/D_{2n}$.
	The sheaf in question is $\mathcal{F} := \mathcal{O}_{Y_1}(L + D_i)$. Computing the degrees of $\mathcal{F}|_{E_j}$, $\mathcal{F}|_Z$, and $\mathcal{F} \otimes \mathcal{O}(-nZ)|_Z$:
	\begin{align*}
		deg(\mathcal{F}|_{E_j}) &= (L + D_i) \cdot E_j.\\
		deg(\mathcal{O}(Z)|_Z) &= -1.\\
		deg(\mathcal{F} \otimes \mathcal{O}(-nZ)|_Z) &= n.
	\end{align*}
	Thus, $H^1(\mathcal{F}|_Z) = 0$, and $H^1(\mathcal{F} \otimes \mathcal{O}(-nZ)|_Z) = 0$.
	Using induction and the exact sequence:
	$$0 \rightarrow \mathcal{F} \otimes \mathcal{O}(-nZ)|_Z \rightarrow \mathcal{F}|_{(n+1)Z} \rightarrow \mathcal{F}|_{nZ} \rightarrow 0,$$
	we have $H^1(\mathcal{F}|_{nZ}) = 0$, for all $n > 0$. By the theorem of formal functions (see Ch.III, Sec.11 of \cite{hartshorne1}), since $R^1(f_2)_*(\mathcal{F})$ is supported on the origin of $\mathbb{C}^2/D_{2n}$ with $f_2^{-1}(0) = Z$, this implies the lemma.
\end{proof}

\begin{remark}
	The lemma can also be applied if $D_i$ is replaced by any of the boundary divisors $B_1$ or $B_2$ since it will use the very same argument. This is needed for the following description of tautological bundles on the coarse moduli space.
\end{remark}

\begin{theorem}\label{coarsetaut}
	
	The tautological bundles on the coarse moduli space $Y_1 = \mathbb{Z}_n\operatorname{-Hilb}(\mathbb{C}^2)/\mathbb{Z}_2$ are described by the following:
	
	\begin{center}
		\captionof{table}{The Tautological Sheaves for odd $n$}
		\begin{tabular}{|c c c|} 
			\hline
			Tautological Bundle & Description & Chern Class \\ [0.5ex] 
			\hline
			$\mathcal{R}_{\rho_0}$ & $\mathcal{O}_{Y_1}$ & $0$ \\ 
			\hline
			$\mathcal{R}_{\rho_0'}$ & $\mathcal{O}_{Y_1}(L)$ & $L$ \\
			\hline
			$\mathcal{R}_{\rho_i}$ (rank 2) & $\mathcal{R}_{\rho_i} = \mathcal{O}_{Y_1} \oplus \mathcal{O}_{Y_1}(D_i + L)$ & $D_i + L$ \\
			\hline
		\end{tabular}
		
		\captionof{table}{The Tautological Sheaves for even $n$}
		\begin{tabular}{|c c c|} 
			\hline
			Tautological Bundle & Description & Chern Class \\ [0.5ex] 
			\hline
			$\mathcal{R}_{\rho_0}$ & $\mathcal{O}_{Y_1}$ & $0$ \\ 
			\hline
			$\mathcal{R}_{\rho_0'}$ & $\mathcal{O}_{Y_1}(L)$ & $L$ \\
			\hline
			$\mathcal{R}_{\rho_i}$ (rank 2) & $\mathcal{R}_{\rho_i} = \mathcal{O}_{Y_1} \oplus \mathcal{O}_{Y_1}(D_i + L)$ & $D_i + L$ \\
			\hline
			$\mathcal{R}_{\rho_{(n/2)}}$ & $\mathcal{O}_{Y_1}(B_1 + L)$ & $B_1 + L$ \\
			\hline
			$\mathcal{R}_{\rho_{(n/2)}'}$ & $\mathcal{O}_{Y_1}(B_2 + L)$ & $B_2 + L$ \\
			\hline
		\end{tabular}
	\end{center}
\end{theorem}

\begin{proof}
	Statement (3) of Lemma \ref{finp} describes the tautological sheaves $\mathcal{R}_{\rho_0}$ and $\mathcal{R}_{\rho_0'}$. Lemma \ref{split} implies the description of the rank $2$ tautological bundles $\mathcal{R}_{\rho_i}$.
	
	Using the similar exact sequence as applied in Lemma \ref{split}:
	$$0 \rightarrow O_{X_1} \rightarrow O_{X_1}(\tilde{B}_1) \oplus O_{X_1}(\tilde{B}_2) \rightarrow O_{X_1}(\tilde{B}_1 + \tilde{B}_2) \otimes \delta \rightarrow 0.$$
	Taking the push-forward $p_*$ and $\mathbb{Z}_2$-invariants gives the exact sequence: 
	$$0 \rightarrow O_{Y_1} \rightarrow [p_*(O_{X_1}(\tilde{B}_1) \oplus O_{X_1}(\tilde{B}_2))]^{\mathbb{Z}_2} \rightarrow O_{Y_1} \rightarrow 0.$$
	Because $\operatorname{Ext}^1(O_{Y_1}, O_{Y_1}) = H^1(Y_1, O_{Y_1}) = H^1(\mathbb{C}^2/G, O_{\mathbb{C}^2/G}) = 0$, this implies that $[p_*(O_{X_1}(\tilde{B}_1) \oplus O_{X_1}(\tilde{B}_2))]^{\mathbb{Z}_2} = (O_{Y_1})^{\oplus 2}$. Furthermore, by the statement (1) of Lemma \ref{finp}, we have the corresponding descriptions for $\mathcal{R}_{\rho_{(n/2)}}$ and $\mathcal{R}_{\rho_{(n/2)}'}$.
\end{proof}

\begin{remark}
	\indent
	\begin{enumerate}
		\item A consequence of Theorem \ref{mainthm2} is that the rank $2$ tautological sheaves can be expressed as an extension of two line bundles on the stack. Because of the structure of the line bundles which are torsion elements of the Picard group of the stack as given in the Lemma \ref{tautdes}, certainly, it is not generated by global sections, which is an obstruction in writing an exact sequence similar to Theorem 1.1 of \cite{artverd}.
		\item Lemma \ref{split} implies that over the coarse moduli space $Y_1$, the rank two tautological bundles split, whereas over the stack $[X_1/\mathbb{Z}_2]$, the tautological bundle does not split. Furthermore, there is a parallel description of \cite{artverd} that can be applied to stacks (especially when identified as $\mathbb{Z}_2$-equivariant sheaves on $X_1$). This makes the stack a more appropriate venue to establish the McKay correspondence for the dihedral group than the coarse moduli space.
		\item The rank two tautological bundles $\tilde{\mathcal{R}}_{\rho_i}$ ($\rho_i \neq \rho_0, \rho_0', \rho_{n/2}, \rho_{n/2}'$) can be realized as an extension via the exact sequences: 
		\begin{align*}
			0 &\rightarrow p^*(\mathcal{R}_{\rho_i}) \rightarrow \tilde{\mathcal{R}}_{\rho_i} \rightarrow O_{\tilde{B}} \rightarrow 0 \mbox{ and } \\
			0 &\rightarrow O_{Y_{max}} \rightarrow p_*(O_{\mathbb{Z}_n\operatorname{-Hilb}(\mathbb{C}^2)}) \rightarrow O_{Y_{max}}(L) \rightarrow 0
		\end{align*} where $\tilde{\mathcal{R}}_{\rho_i} = O_{X_1}(\tilde{D_i}) \oplus O_{X_1}(g \cdot \tilde{D_i})$ and $p_*(\tilde{\mathcal{R}}_{\rho_i}) = O_{Y_{max}} \oplus O_{Y_{max}}(D_i) \otimes O_{Y_{max}}(L) \otimes \delta$, where $D_i$ is any transversal to the exceptional divisor $E_i$ not intersecting $E_j, j \neq i$, and not intersecting the boundary divisors; and correspondingly for $\tilde{D_i}$, which is a transversal to $\tilde{E_i}$ not intersecting on other exceptional divisors and not intersecting the boundary divisors. The exact sequence is obtained by evaluating Chern classes of each sheaves with the fact that the boundary divisors $\tilde{B}$ being isomorphic to $\mathbb{A}^1$.
		
		The extension is classified by:
		\begin{align*}
			Ext^1(O_{\tilde{B}}, p^*(\mathcal{R}_{\rho_i})) &\cong Ext^1(O_{\tilde{B}}, O(\tilde{D_i} + \tilde{D}_{n-i} - \tilde{B}))\\ 
			&\cong H^0(\tilde{B}, O(\tilde{D_i} + \tilde{D}_{n-i})|_{\tilde{B}})\\
			&\cong H^0(\tilde{B}, O_{\tilde{B}}^{\oplus 2})\\
			&\cong (\mathbb{C}[\tilde{B}])^{\oplus 2}
		\end{align*}
	\end{enumerate}
\end{remark}

\section{The McKay Correspondence via the Top and the Socles}
In this section, we reveal that the top and socles (defined in Definition \ref{topsoc}) can be described over the stacks, but not enough to construct such a correspondence over the coarse moduli space.

We refer to the previous section for the functors $\Phi$ and $\Psi$.

We wish to evaluate $\Phi(O_0 \otimes \rho_i^*)$ for $\rho_i \in \operatorname{Irr}(D_{2n})$.

By referring to the results in \cite{art1ishii} regarding the functor $\Phi^{\mathbb{Z}_n}$, the computation of $\Phi(O_0 \otimes \rho_i^*)$ reduces to the task of determining $\mathbb{Z}_2$-equivariant structures on $\Phi^{\mathbb{Z}_n}(\operatorname{for}(O_0 \otimes \rho_i^*))$.

We define the representations (as a pair) $(P_i, \rho_i)$ as a representation of $D_{2n}$ and $(D_i, \epsilon_i := (\rho_i)|_{\mathbb{Z}_n})$ as its corresponding restriction to the cyclic group $\mathbb{Z}_n$. Furthermore, we consider the exact sequence $0 \rightarrow D_i \hookrightarrow P_i \rightarrow P_i/D_i \rightarrow 0$, so that $(P_i/D_i, \delta_i)$ is a representation of $D_{2n}/\mathbb{Z}_n \cong \mathbb{Z}_2$. We have the following:

\begin{proposition}
	Using the Fourier-Mukai transform \begin{align*}
		\Phi: D^{D_{2n}}(\mathbb{C}^2) &\rightarrow D^{\mathbb{Z}_2}(X_1) \\
		\mathfrak{G} &\mapsto Rp_{X_1*} (p_{\mathbb{C}^2}^*(\mathfrak{G}) \otimes O_{\mathcal{Z}}),
	\end{align*} the following images of structure sheaves at the origin has the following images:
	\begin{align*}
		\Phi(O_0 \otimes \rho_0^*) &= O_F \otimes \delta_0 \mbox{ (for any $n$)}\\
		\Phi(O_0 \otimes \rho_0'^*) &= O_F \otimes \delta_0' \mbox{ (for any $n$)}\\
		\Phi(O_0 \otimes \rho_j^*) &= (O_{\tilde{E}_j}(-1) \oplus O_{\tilde{E}_{n-j}}(-1))[1] \mbox{ (for any $n$, and $j \neq \frac{n}{2}, \frac{n-1}{2}$)}\\
		\Phi(O_0 \otimes \rho_{n/2}^*) &= O_{\tilde{E}_{n/2}}(-\tilde{B}_1)[1] \mbox{ (for even $n$)}\\
		\Phi(O_0 \otimes \rho_{n/2}'^*) &= O_{\tilde{E}_{n/2}}(-\tilde{B}_2)[1] \mbox{ (for even $n$)}\\
		\Phi(O_0 \otimes \rho_{(n-1)/2}^*) &= O_{\tilde{E}_{(n-1)/2}}\left(-\tilde{B}_3 \right) [1] \mbox{ (for odd $n$)}
	\end{align*}
	
	where we refer to (\ref{excd}), (\ref{eqb1}), and (\ref{eqb2}) for the definition of $\tilde{E}$ and $\tilde{B}$; and $F$ is the fundamental cycle $\Sigma \tilde{E_i}$. Also, the group $D_{2n}/\mathbb{Z}_n \cong \mathbb{Z}_2$ fixes the subschemes $F$, and $\tilde{E_i} \cup \tilde{E}_{n-i}$; thus, $\mathbb{Z}_2$ acts on the line bundles $O_F$, $O_{\tilde{E}_i}(-1) \oplus O_{\tilde{E}_{n-i}}(-1)$, and $O_{\tilde{E}_{n/2}}(-1)$.
\end{proposition}

\begin{proof}
	From the commutative diagram above,
	\begin{align*}
		\Phi^{\mathbb{Z}_n}(\operatorname{for}(O_0 \otimes \rho_0)) &= \Phi^{\mathbb{Z}_n}(O_0 \otimes \epsilon_0) = O_F\\
		\Phi^{\mathbb{Z}_n}(\operatorname{for}(O_0 \otimes \rho'_0)) &= \Phi^{\mathbb{Z}_n}(O_0 \otimes \epsilon_0) = O_F\\
		\Phi^{\mathbb{Z}_n}(\operatorname{for}(O_0 \otimes \rho_j)) &= \Phi^{\mathbb{Z}_n}(O_0 \otimes \epsilon_j \oplus O_0 \otimes \epsilon_{n-j}) = O_{\tilde{E_j}}(-1)[1] \oplus O_{\tilde{E}_{n-j}}(-1)[1]\\
		\Phi^{\mathbb{Z}_n}(\operatorname{for}(O_0 \otimes \rho_{n/2})) &= \Phi^{\mathbb{Z}_n}(O_0 \otimes \epsilon_{n/2}) = O_{\tilde{E}_{n/2}}(-1)[1]\\
		\Phi^{\mathbb{Z}_n}(\operatorname{for}(O_0 \otimes \rho'_{n/2})) &= \Phi^{\mathbb{Z}_n}(O_0 \otimes \epsilon_{n/2}) = O_{\tilde{E}_{n/2}}(-1)[1].
	\end{align*}
	The sheaf $O_{\tilde{E_j}}(-1)[1] \oplus O_{\tilde{E}_{n-j}}(-1)[1]$ is certainly a $\mathbb{Z}_2$-equivariant sheaf. Thus, it remains to determine $\mathbb{Z}_2$ equivariant structures on $O_F$ and $O_{\tilde{E}_{n/2}}(-1)$.
	
	By considering the canonical isomorphism $\mu_g: g^*O_F \xrightarrow{\sim} O_{g^{-1}(F)} = O_F$, and given the $G$-sheaf $\lambda_g^{O_F}: O_F \rightarrow g^*(O_F)$, the composition given by $\mu_g \circ \lambda_g^{O_F} \in \operatorname{Hom}(O_F, O_F) = \mathbb{C}$. Thus, $\mu_g \circ \lambda_g = c$, so that $\lambda_g = c \mu_g^{-1}$. Using the condition for $G$-sheaves, then $c = \pm 1$ which in turn determines the $\mathbb{Z}_2$-equivariant sheaves.
	
	Similarly, this gives the $\mathbb{Z}_2$ equivariant structures on $O_{\tilde{E}_{n/2}}(-1)$, as
	\begin{align*}
		\operatorname{Hom}(O_{\tilde{E}_{n/2}}(-1), g^*(O_{\tilde{E}_{n/2}}(-1))) &\cong \operatorname{Hom}(O_{\tilde{E}_{n/2}}, g^*(O_{\tilde{E}_{n/2}}))\\ 
		&\cong \operatorname{Hom}(O_{\tilde{E}_{n/2}}, O_{\tilde{E}_{n/2}})\\
		&\cong \mathbb{C}.
	\end{align*}
	 Unfortunately, as a $\mathbb{Z}_2$-equivariant sheaf, $O_{\tilde{E}_{n/2}}(-1)$ must be a fixed locus, which can only be any of the $\tilde{B}_i$. Since $O_{\tilde{E}_{n/2}}(-\tilde{B}_1)$ (resp. $-\tilde{B}_2$) is invariant under the action of $\rho_{n/2}$ (resp. $\rho'_{n/2}$), this completes the description of the Fourier-Mukai images of the skyscraper sheaves.
\end{proof}

Because there is a stacky structure on the fixed points of $\mathbb{Z}_2$, we restate the proposition above in terms of coherent sheaves on the global quotient stack $[X_1/\mathbb{Z}_2]$.

For the next proposition, we define some notations:
We consider the morphism of schemes $p: X_1 \rightarrow Y_1$  and stacks $\pi : [X_1/\mathbb{Z}_2] \rightarrow Y_1$. We define the following closed substacks on $[X_1/\mathbb{Z}_2]$:
\begin{center}
	\begin{equation}\label{stackexc}
		\begin{tabular}{ c c }
			$\mathcal{E}_i := [p(\tilde{E}_i \cup \tilde{E}_{n-i})/\mathbb{Z}_2] \cong p(\tilde{E_i} \cup \tilde{E}_{n-i})/\mathbb{Z}_2  \left(i \neq \frac{n-1}{2}, \frac{n}{2}, \frac{n+1}{2} \right)$ & $\mathcal{F} := [p(F)/\mathbb{Z}_2]$\\
			$\mathcal{E}_{(n-1)/2} := [p(\tilde{E}_{(n-1)/2} \cup \tilde{E}_{(n+1)/2})/\mathbb{Z}_2] =: \mathcal{E}_{(n+1)/2}$ & $\mathcal{B}_1 := [p(\tilde{B}_1)/\mathbb{Z}_2]$\\
			$\mathcal{E}_{n/2} := [p(\tilde{E}_{n/2})/\mathbb{Z}_2]$ & $\mathcal{B}_2 := [p(\tilde{B}_2)/\mathbb{Z}_2]$\\
			& $\mathcal{B}_3 := [p(\tilde{B}_3)/\mathbb{Z}_2]$
		\end{tabular}
	\end{equation}
\end{center}
\begin{remark}
	The exceptional divisors on the stack $\mathcal{E}$ are smooth except for $\mathcal{E}_{(n-1)/2}$. The exceptional divisor $\mathcal{E}_{(n-1)/2}$ is not smooth because the fixed point of $\mathbb{Z}_n\operatorname{-Hilb}(\mathbb{C}^2)$ lies on the intersection of two distinct exceptional divisors.
\end{remark}

\begin{theorem}\label{stackfm}
	Using the Fourier-Mukai transform \begin{align*}
		\Phi: D^{D_{2n}}(\mathbb{C}^2) &\rightarrow D^{\mathbb{Z}_2}(X_1) \cong D([X_1/\mathbb{Z}_2]) \\
		\mathfrak{G} &\mapsto Rp_{X_1*} (p_{\mathbb{C}^2}^*(\mathfrak{G}) \otimes O_{\mathcal{Z}}),
	\end{align*} the following images of structure sheaves at the origin has the images on the quotient stack $[X_1/\mathbb{Z}_2]$:
	\begin{align*}
		\Phi(O_0 \otimes \rho_0^*) &= O_{\mathcal{F}} \mbox{ (for any $n$)}\\
		\Phi(O_0 \otimes \rho_0'^*) &= O_{\mathcal{F}}(\mathcal{B}_1 - \mathcal{B}_2) \mbox{ (for any $n$)}\\
		\Phi(O_0 \otimes \rho_j^*) &= O_{\mathcal{E}_j}(-1)[1] \mbox{ (for any $n$, and $j \neq n/2, (n-1)/2$)}\\
		\Phi(O_0 \otimes \rho_{n/2}^*) &= O_{\mathcal{E}_{n/2}}\left( -\mathcal{B}_1 \right) [1] \mbox{ (for even $n$)}\\
		\Phi(O_0 \otimes \rho_{n/2}'^*) &= O_{\mathcal{E}_{n/2}}\left( -\mathcal{B}_2 \right) [1] \mbox{ (for even $n$)}\\
		\Phi(O_0 \otimes \rho_{(n-1)/2}^*) &= O_{\mathcal{E}_{(n-1)/2}}\left( -\mathcal{B}_3 \right) [1] \mbox{ (for odd $n$)}.
	\end{align*}
\end{theorem}

\begin{proof}
	This is a restatement in terms of coherent sheaves on the global quotient stack. Furthermore, the correspondence for $O_F \otimes \delta_i$ can be found in Lemma \ref{tautdes}.
\end{proof}

\begin{remark}
	Take the even $n$ case and the divisor $\mathcal{E}_{n/2}$ for instance. The following justification uses the root stack construction whose introduction and details to the said concepts can be found in \cite{cadman}, \cite{agv}, and \cite{olsson}.
	
	Since $p(\tilde{E}_{n/2})/\mathbb{Z}_2$ is smooth, we can realize $\mathcal{E}_{n/2}$ as the $2$nd root stack 
	$$\mathcal{E}_{n/2} := \sqrt{(O_{p(\tilde{E}_{n/2})}(p(\tilde{B}_1) + p(\tilde{B}_2)), 1)/(p(\tilde{E}_{n/2}))},$$
	where we perform the appropriate modification of the definition of the 2nd root stack seen after Remark \ref{resmoth}. Refer to Section 2.2 of \cite{cadman} for further details.
\end{remark}

The setup propositions and theorems earlier in this section will be used to compute the top and socles in the hopes of obtaining a similar description as in the McKay correspondence of Ito-Nakamura \cite{itonak2} for the $\operatorname{SL}(2)$ case and Ishii \cite{art1ishii} for the small $\operatorname{GL}(2)$ case. The top and socles can be (explicitly) computed via their ideals that define a closed subscheme, i.e. $I_y$, to obtain the corresponding quotients $I_y/mI_y$ and $(I_y:m)/I_y$ needed to compute the top and the socle respectively.

The data for the McKay correspondence is given by the following:

\begin{definition}\label{topsoc}
	For a given $G$-constellation $\overline{F}$ on the moduli space of $\theta$-stable $G$-constellations $\mathcal{M}_{\theta}$:
	\begin{align*}
		\operatorname{top}(\overline{F}) &:= \overline{F}/\langle x, y \rangle \overline{F}\\
		\operatorname{socle}(\overline{F}) &:= \{ a \in \overline{F} | \langle x,y \rangle a = 0\}.
	\end{align*}
\end{definition}

Before we state the proposition, the maximal resolution $Y_{max}$ can be realized as a moduli space of $\theta$-stable $G$-constellations $\mathcal{M}_{\theta}$ for some generic stability parameter $\theta$ via the isomorphism $\mathcal{M}_{\theta} \cong Y \xrightarrow{\sim} Y_1 \cong Y_{max}$. The isomorphism of iterated Hilbert schemes to a moduli space of $G$-constellations is justified in Theorem 1.5 of \cite{ishiiitonolla}. The specific stability parameter $\theta$ is computed in Table 5 of \cite{ishiiitonolla}.

\begin{proposition}\label{socmax}
	For a given $D_{2n}$-constellation $\overline{F}$ on the maximal resolution $Y_{max}$,
	\[ \operatorname{top}(\overline{F}) = \begin{cases} 
		\rho_0 \oplus \rho_0' & [\overline{F}] \in Exc(Y_{max} \rightarrow \mathbb{C}^2/G) \\
		0 & otherwise
	\end{cases} \]
	
	\[ \operatorname{socle}(\overline{F}) = \begin{cases} 
		\rho_i & [\overline{F}] \in E_i, [\overline{F}] \notin E_j, i \neq j \\
		\rho_i \oplus \rho_j & [\overline{F}] \in E_i \cap E_j \\
		\rho_{\frac{n}{2} - 1} \oplus \rho_{n/2} \oplus \rho'_{n/2} & [\overline{F}] \in E_{\frac{n}{2} - 1} \cap E_{n/2} \mbox{ ($n$ even)}\\
		\rho_{n/2} \oplus \rho'_{n/2} & [\overline{F}] \in E_{n/2} - E_{\frac{n}{2} - 1} \mbox{ ($n$ even)}\\
	\end{cases}
	\]
	
	where $E_i$ is an exceptional divisor on $Y_{max}$ such that given the projection morphism $p: X_1 \xrightarrow{\pi} Y_1 \cong Y_{max}$, $p^{-1}(E_i) = \tilde{E_i} \cup \tilde{E}_{n-i}$ (refer again to (\ref{excd}) for the definition of $E$).
\end{proposition}

\begin{proof}
	We consider the diagram:
	\[
	\begin{tikzcd}
		& \mathcal{Y} \times \mathbb{C}^2 \arrow[dr,"\pi \times id_{\mathbb{C}^2}"] \\
		\mathbb{Z}_n\operatorname{-Hilb}(\mathbb{C}^2) \times \mathbb{C}^2 \arrow[ur,"f  \times id_{\mathbb{C}^2}"] \arrow[rr,"p  \times id_{\mathbb{C}^2}"] && \mathbb{Z}_n\operatorname{-Hilb}(\mathbb{C}^2)/\mathbb{Z}_2 \times \mathbb{C}^2.
	\end{tikzcd}
	\]
	
	The universal flat family of $\mathbb{Z}_n$-clusters on $\mathbb{Z}_n\operatorname{-Hilb}(\mathbb{C}^2) \times \mathbb{C}^2$ is denoted by $\mathcal{U}_0$.
	
	The universal flat family over the stacks $\tilde{\mathcal{U}}$ and the coarse moduli space $\mathcal{U}$ are defined as follows:
	$$\tilde{\mathcal{U}} = (f  \times id_{\mathbb{C}^2})_*(\mathcal{U}_0); \mbox{          } \mathcal{U} = (\pi \times id_{\mathbb{C}^2})_*(\tilde{\mathcal{U}}).$$
	
	Our main interest here is to know the socles over the fixed point. The non-stacky points follow from the Fourier-Mukai image of the skyscraper sheaf $\mathcal{O}_{\tilde{y}} \oplus \mathcal{O}_{g \cdot \tilde{y}}$ on $\mathbb{Z}_n\operatorname{-Hilb}(\mathbb{C}^2)$.
	
	For a fixed point $\tilde{y}$ in $\mathbb{Z}_n\operatorname{-Hilb}(\mathbb{C}^2)$ under the $\mathbb{Z}_2$-action and $y = p(\tilde{y})$, we consider the exact sequence of skyscraper sheaves over the stack $\mathcal{Y}$:
	$$O \rightarrow \mathcal{O}_{\tilde{y}} \otimes \delta_0' \rightarrow \mathcal{O}_{\pi^{-1}y} \rightarrow \mathcal{O}_{\tilde{y}} \otimes \delta_0 \rightarrow 0$$
	which realizes the skyscraper sheaf $\mathcal{O}_{\pi^{-1}y}$ as a nontrivial extension of two $\mathbb{Z}_2$-equivariant sheaves.
	
	Using the Fourier-Mukai transform $\Psi$, the exact sequence translates to:
	$$O \rightarrow \mathcal{U}_{0, \tilde{y}} \otimes \delta_0' \rightarrow \tilde{\mathcal{U}}_{\pi^{-1}y} \rightarrow \mathcal{U}_{0, \tilde{y}} \otimes \delta_0 \rightarrow 0.$$
	
	By the definition of the universal families, the fiber is realized as: $$\mathcal{U}_x = ((\pi \times id_{\mathbb{C}^2})_* \tilde{\mathcal{U}})_x = \tilde{\mathcal{U}} \otimes_{\mathcal{O}_{\mathcal{X}}} \mathcal{O}_{\pi^{-1}(x)}.$$
	
	Realizing $\mathcal{U}_{\tilde{y}}$ as a $\mathbb{Z}_2$-invariant cluster yields that the socle of $\mathcal{U}_{\tilde{y}} \otimes \delta_0$ is $\rho_{n/2}$ and similarly, the socle of $\mathcal{U}_{\tilde{y}} \otimes \delta_1$ is $\rho_{n/2}'$.
\end{proof}

\begin{remark}\label{opencons}
	We first enumerate the $D_{2n}$-constellations on each open subset of the maximal resolution $Y$. We are mainly interested in the open sets $U_{m+1}'$ and $U_m''$ which cover the exceptional divisor $E_{n/2}$:
	
	Open set: $U_m'' = \operatorname{Spec} \left(\mathbb{C}\left[\frac{(x^m + y^m)^2}{(xy)^{m-1}}, \frac{(x^m - y^m)^2}{(x^m + y^m)^2} \right]\right)$
		
	$\mathbb{Z}_n$-constellation:
	\begin{align*}
		Z_m &= \{1, y, y^2, \cdots, y^{m-1}, x, x^2, \cdots, x^{m-1}, x^m - y^m \}
	\end{align*}
	
	$D_{2n}$-constellation: 
	
	\begin{tabular}[t]{
			| >{\ttfamily\raggedright}p{2.0cm}
			| >{\sffamily\raggedright}p{1.5cm}
			| >{\sffamily\raggedright}p{1.5cm}
			| >{\sffamily\raggedright}p{1.5cm}
			| >{\sffamily\raggedright}p{2.0cm}
			| >{\sffamily}p{\dimexpr\textwidth-12\tabcolsep-5\fboxsep-10cm\relax} |
		}
		\firsthline
		$1$ & $(x,y)$ & $(x^2, y^2)$ & $\cdots$ & $\alpha (x^m - y^m)$ & \\
		\hline
		$\alpha := \frac{x^{m} - y^{m}}{x^{m} + y^{m}}$ & $\alpha (y,-x)$ & $\alpha (y^2,-x^2)$ & $\cdots$ & $x^m - y^m$ & \\
		\hline
	\end{tabular}
	
	Open set: $U_{m+1}' = \operatorname{Spec} \left(\mathbb{C}\left[\frac{(x^m + y^m)^2}{(x^m - y^m)^2}, \frac{(x^m - y^m)^2}{(xy)^{m-1}} \right] \right)$
	
	$\mathbb{Z}_n$-constellation:
	\begin{align*}
		Z_m &= \{1, y, y^2, \cdots, y^{m-1}, x, x^2, \cdots, x^{m-1}, x^m + y^m \}
	\end{align*}
	
	$D_{2n}$-constellation:. 
	
	\begin{tabular}[t]{
			| >{\ttfamily\raggedright}p{2.0cm}
			| >{\sffamily\raggedright}p{1.5cm}
			| >{\sffamily\raggedright}p{1.5cm}
			| >{\sffamily\raggedright}p{1.5cm}
			| >{\sffamily\raggedright}p{2.0cm}
			| >{\sffamily}p{\dimexpr\textwidth-12\tabcolsep-5\fboxsep-10cm\relax} |
		}
		\firsthline
		$1$ & $(x,y)$ & $(x^2, y^2)$ & $\cdots$ & $x^m + y^m$ & \\
		\hline
		$\beta := \frac{x^m + y^m}{x^m - y^m}$ & $\beta (y,-x)$ & $\beta (y^2,-x^2)$ & $\cdots$ & $\beta (x^m + y^m)$ & \\ 
		\hline
	\end{tabular}
	
	We compute the top and the socle by the following: for example in $U_m''$ (and similarly for $U_{m+1}'$), the corresponding open set in $X_1$ is given by $A = \operatorname{Spec}(\mathbb{C}\left[\frac{(x^m + y^m)^2}{(xy)^{m-1}}, \frac{x^m - y^m}{x^m + y^m} \right])$, so that every $\mathbb{Z}_n$-cluster is given by the ideal $I_{a,b} = \langle (x^m + y^m)^2 -a(xy)^{m-1}, (x^m - y^m) - b(x^m + y^m), (x^{2m} - y^{2m}) - ab(xy)^{m-1} \rangle$. Thus, $x \cdot (x^m - y^m) = x^{m+1} - xy^m = 0$ (and also $x \cdot \alpha (x^m - y^m) = x \cdot b(x^m - y^m) = x \cdot b^2(x^m + y^m) = 0$, and similarly for the multiplication by $y$) by interpreting the ideals as a $\mathbb{Z}_n$-cluster defined by the quotient $ \mathbb{C}[x,y]/\langle p_i x^m - q_i y^m, x^{m+1}, y^{m+1}, xy \rangle$. 
	
\end{remark}

	Based on the character table, given that the natural presentation $\rho_{nat}$ is isomorphic to its dual, i.e. $\rho_{nat} \cong \rho_{nat}^{\lor}$, the McKay quiver for $D_{2n}$ is given by the following (with the first quiver for odd $n$ and second quiver for even $n$ respectively), which is identical to the McKay quiver of the binary dihedral group in the $\operatorname{SL}(2)$ case:
	
\begin{center}
\begin{tikzcd}[row sep=1em,column sep=1em]
	\circ_{\rho_0} \arrow[dr] & & & &  \\
	& \circ_{\rho_1} \arrow[ul]\arrow[dl]\arrow[r] & \circ_{\rho_2} \arrow[r]\arrow[l] & \cdots \arrow[r]\arrow[l] & \circ_{\rho_{\frac{n-1}{2}}} \arrow[l] \\
	\circ_{\rho_0'} \arrow[ur]& & & &
\end{tikzcd} 

\begin{tikzcd}[row sep=1em,column sep=1em]
	\circ_{\rho_0} \arrow[dr] & & & & & \circ_{\rho_{n/2}} \arrow[dl] \\
	& \circ_{\rho_1} \arrow[ul]\arrow[dl]\arrow[r] & \circ_{\rho_2} \arrow[r]\arrow[l] & \cdots \arrow[r]\arrow[l] & \circ_{\rho_{\frac{n}{2} - 1}} \arrow[ur]\arrow[dr]\arrow[l]\\
	\circ_{\rho_0'} \arrow[ur]& & & & & \circ_{\rho_{n/2}'} \arrow[ul]
\end{tikzcd}
\end{center}
Consider the following bijection:

\begin{align*}
	E_i &\mapsto \rho_i\\
	E_{n/2} &\mapsto \rho_{n/2} \oplus \rho_{n/2}' \mbox{ ($n$ even)}.
\end{align*}

It is imperative to comment on the possible McKay correspondence via top and socles. Compared to the representation of a binary dihedral group as a small finite subgroup of $\operatorname{GL}(2, \mathbb{C})$, particularly in the even case, the exceptional divisor $E_{n/2}$ corresponds to the two-dimensional representation $\rho_{n/2} \oplus \rho_{n/2}'$. This is because the socle failed to separate the two 1-dimensional representations $\rho_{n/2}$ and $\rho_{n/2}'$. This tells us that such a correspondence is not the ‘ideal' correspondence on the coarse moduli space.

\begin{remark}
	Considering the dual graph of exceptional divisors on $Y_{max}$, possibly referring to the computation of socles over the coarse moduli space, that there is a bijection between exceptional divisors and representations of the group $G$. More precisely:
	
	For odd $n$, we consider such bijection: $E_i \mapsto \rho_i$.
	
	However, for even $n$, the mapping given by 
	\begin{align*}
		E_i &\mapsto \rho_i\\
		E_{n/2} &\mapsto \rho_{n/2} \oplus \rho_{n/2}'
	\end{align*} 
	gives a bijection between two-dimensional irreducible representations of $G$ and the exceptional divisors whose self-intersection number is $-2$; and the exceptional divisor corresponding to the two-dimensional decomposable representation has self-intersection number $-1$.
	
	As seen in the bijection:
	
	For odd $n$, there are $\frac{n-1}{2}$ irreducible representations which correspond to the exceptional divisors of the maximal resolution.
	
	For even $n$, there are $\frac{n}{2} - 1$ irreducible representations which correspond to the exceptional divisors of the maximal resolution.
\end{remark}

\begin{theorem}\label{mainthm3}
	For a given $D_{2n}$-constellation $\overline{F}_{st}$ on the exceptional divisors over the quotient stack $\mathcal{Y}$ (refer to (\ref{stackexc}) for the definitions),
	\[ \operatorname{top}(\overline{F}_{st}) = \rho_0 \oplus \rho_0' \]
	\[ \operatorname{socle}(\overline{F}_{st}) = \begin{cases} 
		\rho_i & [\overline{F}_{st}] \in \mathcal{E}_i, [\overline{F}_{st}] \notin \mathcal{E}_j, i \neq j \\
		\rho_i \oplus \rho_j & [\overline{F}_{st}] \in \mathcal{E}_i \cap \mathcal{E}_j \\
		\rho_{\frac{n}{2} - 1} \oplus \rho_{n/2} \oplus \rho'_{n/2} & [\overline{F}_{st}] \in \mathcal{E}_{\frac{n}{2} - 1} \cap \mathcal{E}_{n/2} \mbox{ ($n$ even)}\\
		\rho_{n/2} \oplus \rho'_{n/2} & [\overline{F}_{st}] \in \mathcal{E}_{n/2} - (\mathcal{E}_{\frac{n}{2} - 1} \cup \{\mathcal{B}_1, \mathcal{B}_2\}) \mbox{ ($n$ even)}\\
		\rho'_{n/2} & [\overline{F}_{st}] = \mathcal{B}_1 \mbox{ ($n$ even)}\\
		\rho_{n/2} & [\overline{F}_{st}] = \mathcal{B}_2 \mbox{ ($n$ even)}.\\
	\end{cases}
	\]
\end{theorem}

\begin{proof}
	Now, we consider the derived equivalences $\Psi$ and $\Phi$ and compute in the level of the global quotient stacks. We use the derived equivalence in (7.1) of \cite{art1ishii} and let $\phi: D(\mathcal{Y}) \xrightarrow{\sim} D^{\mathbb{Z}_2}(X_1)$ be the derived equivalence between the coherent sheaves on $\mathcal{Y}$ and $\mathbb{Z}_2$-equivariant sheaves on $X_1$:
	\begin{align*}
		\operatorname{Hom}_{D(\mathcal{Y})}^k(\Phi(O_0 \otimes \rho_i^*), O_{\bold{y}}) &\cong \operatorname{Hom}_{D^{\mathbb{Z}_2}(X_1)}^k(\phi(\Phi(O_0 \otimes \rho_i^*)), \phi(O_{\bold{y}}))\\
		&\cong \operatorname{Hom}_{D^G(\mathbb{C}^2)}^k(O_0 \otimes \rho_i^*, \Psi(O_{\bold{y}})).
	\end{align*}
	It is imperative to notice here that $\phi(O_{\bold{y}})$ depends on whether the point is stacky or not.
	\[ \phi(O_{\bold{y}}) = \begin{cases} O_y & \mbox{ if $\bold{y}$ is a stacky point}\\
		O_y \oplus O_{g \cdot y} & \mbox{ if $\bold{y}$ is not a stacky point}
	\end{cases}\]
	Once again, from (7.2) of \cite{art1ishii}, where $Z_y$ is the subscheme of $\mathbb{C}^2$ corresponding to $\bold{y}$ and $\mathcal{F}^{\lor} := \operatorname{RHom}_{O_{\mathbb{C}^2}}(\mathcal{F},O_{\mathbb{C}^2})$ is the derived dual:
	\[ \Psi(O_{\bold{y}}) = \begin{cases} O_{Z_y}^{\lor} \otimes K_{\mathbb{C}^2}[2] & \mbox{ if $\bold{y}$ is a stacky point}\\
		(O_{Z_y}^{\lor} \oplus O_{Z_{g \cdot y}}^{\lor}) \otimes K_{\mathbb{C}^2}[2] & \mbox{ if $\bold{y}$ is not a stacky point}.
	\end{cases}\]
	By Serre duality:
	\begin{align*}
		\operatorname{Hom}_{D^G(\mathbb{C}^2)}^k(O_0 \otimes \rho_i^*, \Psi(O_{\bold{y}})) &= G\operatorname{-Hom}_{\mathbb{C}^2}^k(O_0 \otimes \rho_i^*, \Psi(O_{\bold{y}}))\\
		&\cong G\operatorname{-Hom}_{\mathbb{C}^2}^{2-k}(O_0 \otimes \rho_i \otimes det(\rho_{nat}), O_{Z_y} (\oplus O_{Z_{g \cdot y}}))\\
		&\cong G\operatorname{-Hom}_{\mathbb{C}^2}^{2-k}(O_0 \otimes \rho_i \otimes \rho_0', O_{Z_y} (\oplus O_{Z_{g \cdot y}})).
	\end{align*}
	We are now in the position to compile each of the equivalences:
	\begin{align*}
		\operatorname{Hom}_{D(\mathcal{Y})}^k(O_{\mathcal{F}}, O_{\bold{y}}) &\cong G\operatorname{-Hom}_{\mathbb{C}^2}^{2-k}(O_0 \otimes \rho_0', O_{Z_y} (\oplus O_{Z_{g \cdot y}}))\\
		&\mbox{ (for $i = 0$)}\\
		\operatorname{Hom}_{D(\mathcal{Y})}^k(O_{\mathcal{F}}(\mathcal{B}_1 - \mathcal{B}_2), O_{\bold{y}}) &\cong G\operatorname{-Hom}_{\mathbb{C}^2}^{2-k}(O_0 \otimes \rho_0, O_{Z_y} (\oplus O_{Z_{g \cdot y}}))\\
		&\mbox{ (for $i = 0'$)}\\
		\operatorname{Hom}_{D(\mathcal{Y})}^k(O_{\mathcal{E}_i}(-1), O_{\bold{y}}) &\cong G\operatorname{-Hom}_{\mathbb{C}^2}^{2-k}(O_0 \otimes \rho_i, O_{Z_y} (\oplus O_{Z_{g \cdot y}}))\\
		&\mbox{ (for $1 \leq i \leq m-1$)}\\
		\operatorname{Hom}_{D(\mathcal{Y})}^k(O_{\mathcal{E}_m}(-\mathcal{B}_3), O_{\bold{y}}) &\cong G\operatorname{-Hom}_{\mathbb{C}^2}^{2-k}(O_0 \otimes \rho_m, O_{Z_y} (\oplus O_{Z_{g \cdot y}}))\\
		&\mbox{ (for $i = m$; $n$ odd)}\\
		\operatorname{Hom}_{D(\mathcal{Y})}^k(O_{\mathcal{E}_m}(-\mathcal{B}_2), O_{\bold{y}}) &\cong G\operatorname{-Hom}_{\mathbb{C}^2}^{2-k}(O_0 \otimes \rho_{n/2}, O_{Z_y} (\oplus O_{Z_{g \cdot y}}))\\
		&\mbox{ (for $i = n/2$; $n$ even)}\\
		\operatorname{Hom}_{D(\mathcal{Y})}^k(O_{\mathcal{E}_m}(-\mathcal{B}_1), O_{\bold{y}}) &\cong G\operatorname{-Hom}_{\mathbb{C}^2}^{2-k}(O_0 \otimes \rho_{n/2}', O_{Z_y} (\oplus O_{Z_{g \cdot y}}))\\
		&\mbox{ (for $i = n/2'$; $n$ even)}.\\
	\end{align*}
	
	Setting $k=2$ in the above equivalences, we obtain a more refined structure of the socles over the quotient stack.
\end{proof}

\begin{remark}
	We can define the top and the socle of a $D_{2n}$-constellation on the stacky locus on the quotient stack as the $1/2$ of the $D_{2n}$-constellation on the coarse moduli space. In simpler terms, this is simply the $\mathbb{Z}_2$-invariant $\mathbb{Z}_n$-cluster corresponding to the fixed point of $\mathbb{Z}_n\operatorname{-Hilb}$ under the action of $\mathbb{Z}_2$. This is reflected in the proof of the proposition.
\end{remark}

\begin{example}
	In light of the previous remark, we end this paper by giving an explicit example for the dihedral group of order $8$ ($n=4$).
	
	Over the non-stacky locus, we can refer to the computations in \ref{opencons} (and all of the other open sets) and let $m = 2$.
	
	Over the stacky (irreducible) locus, we consider first $\overline{F}_{st} = O_{\mathbb{C}^2}/\langle x^3, y^3, xy, x^2 + y^2 \rangle$ (corresponding to the point $I_{n/2}(1:-1)$ in $\mathbb{Z}_n\operatorname{-Hilb}(\mathbb{C}^2)$). The $\mathbb{Z}_4$-cluster has $\mathbb{C}[\mathbb{Z}_4]$-basis $\{1, x, y, x^2\}$. Certainly, the socle of this $D_8$-constellation is generated by $x^2$. But by the relationship, $x^2 + y^2 = 0$ in this constellation, $x^2 \otimes \delta_1 = -y^2 \otimes \delta_1$, making the socle $\rho_2'$, where $\delta_1$ is the non-trivial representation in the proof of \ref{socmax}. Similarly, $\overline{F}_{st} = O_{\mathbb{C}^2}/\langle x^3, y^3, xy, x^2 - y^2 \rangle$ (corresponding to the point $I_{n/2}(1:1)$ in $\mathbb{Z}_n\operatorname{-Hilb}(\mathbb{C}^2)$) has socle $\rho_2$.
\end{example}

\section*{Acknowledgements}
The author expresses his greatest and deepest gratitude to Prof. Akira Ishii for the numerous discussions that were essential in the preparation of this article. The author also thanks the unknown referee for the comments and suggestions towards the improvement of the article. The author is partially supported by the grants JSPS KAKENHI Grant No. 19K03444 and JST SPRING Grant No. JPMJSP2125 (THERS Make New Standards Program for the Next Generation Researchers).

\end{document}